\documentclass[12pt,a4paper]{amsart}

\usepackage{amssymb,color,currfile,pdfpages,mathtools}

\usepackage[english]{babel}
\usepackage{verbatim,here}
\usepackage[T1]{fontenc}
\usepackage{epstopdf}
\usepackage{floatflt,graphicx}
\usepackage{color,xcolor}
\usepackage{subcaption}
\usepackage{orcidlink}
\usepackage{hyperref}
\usepackage{enumerate}
\usepackage{animate}
\usepackage{a4wide}
\usepackage{amsmath, amssymb}
\usepackage{amsthm}
\usepackage{float}
\usepackage{pdfpages}
\usepackage{algorithm}
\usepackage{algpseudocode}
\newcounter{minutes}
\divide\time by 60
\newcounter{hours}
\multiply\time by 60 \addtocounter{minutes}{-\time}

\font\fFt=eusm10 
\font\fFa=eusm7  
\font\fFp=eusm5  
\def\K{\mathchoice
{\hbox{\,\fFt K}}
{\hbox{\,\fFt K}}
{\hbox{\,\fFa K}}
{\hbox{\,\fFp K}}}
\newcommand{\M}{\mathsf{M}}

\dedicatory{}
\commby{}
\theoremstyle{plain}
\newtheorem{thm}[equation]{Theorem}
\newtheorem{cor}[equation]{Corollary}
\newtheorem{lem}[equation]{Lemma}

\newtheorem{prop}[equation]{Proposition}

\theoremstyle{definition}

\theoremstyle{remark}
\newtheorem{rem}[equation]{Remark}

\newtheorem{nonsec}[equation]{}

\numberwithin{equation}{section}

\newcommand{\beq}{\begin{equation}}
\newcommand{\eeq}{\end{equation}}
\newcommand{\ben}{\begin{enumerate}}
\newcommand{\een}{\end{enumerate}}
\newcommand{\bequu}{\begin{eqnarray*}}
\newcommand{\eequu}{\end{eqnarray*}}
\newcommand{\bequ}{\begin{eqnarray}}
\newcommand{\eequ}{\end{eqnarray}}

\newcommand{\sh}{\,\textnormal{sh}}
\newcommand{\ch}{\,\textnormal{ch}}
\renewcommand{\th}{\,\textnormal{th}}

\begin{document}
\thispagestyle{empty}
\def\thefootnote{}

%
\title[Hilbert metric and H\"older continuity of quasiregular mappings]
{Hilbert metric and H\"older continuity of quasiregular mappings}

\author[\c{S}. Alt\i nkaya]{\c{S}ahsene Alt\i nkaya}
\address{Department of Mathematics and Statistics,
	University of Turku,\newline FI-20014 Turku, 
	Finland\\
	\url{https://orcid.org/0000-0002-7950-8450}}
\email{sahsene.altinkaya@utu.fi}

\author[M. FUJIMURA]{Masayo FUJIMURA}
\address{Department of Mathematics,
	National Defense Academy of Japan, \newline Yokosuka, Japan \\
	\url{https://orcid.org/0000-0002-5837-8167}}
\email{masayo@nda.ac.jp}

\author[M. MOCANU]{Marcelina Mocanu}
\address{Department of Mathematics and Informatics,
	"Vasile Aleksandi" University of Bac\v{a}u, Bac\v{a}u,  Romania \\
	\url{ https://orcid.org/0000-0002-0192-8180}}
	  \email{mmocanu@ub.ro}
\author[M. Vuorinen]{Matti Vuorinen}   
\address{Department of Mathematics and Statistics,
	University of Turku, \newline
	FI-20014 Turku,
	Finland\\ \url{https://orcid.org/0000-0002-1734-8228}} 
\email{vuorinen@utu.fi}

\date{}

\begin{abstract}
We prove several formulas for the Hilbert metric in the unit disk and
 apply these results to
study quasiregular mappings of the unit disk
$\mathbb{B}^2$ onto a bounded convex domain $D\,.$ The main result deals with the H\"older continuity of 
these mappings with respect to  Hilbert metrics of  $\mathbb{B}^2$ and $D.$ Also several open problems are formulated.
\end{abstract}



\maketitle


\footnotetext{\texttt{{\tiny File:~\jobname .tex, printed: \number\year-%
\number\month-\number\day, \thehours.\ifnum\theminutes<10{0}\fi\theminutes}}}
\makeatletter

\makeatother



\section{Introduction}
It is well known that the hyperbolic metric and similar metrics play an important role in geometric function theory, see \cite{bm, h,gh, hkv}. In his extensive study on metrics occuring
in geometric function theory,  A. Papadopoulos \cite{p} provides a list of 
twelve metrics that are frequently used in the field \cite[pp.42-49]{p}.
Metrics of hyperbolic type have been applied as a tool and studied intensively by several authors \cite{b3, bm,fkv,gh, himps, kr,p, sop,rv}. All this work leads, on one hand, to the question of comparison inequalities between metrics and, on the other hand, to the challenge of viewing
the classical topics from new points of view.

D. Hilbert  \cite{hilb} introduced one of these metrics defined in a bounded convex domain $D \subset \mathbb{R}^2$ for two distinct points $a,b \in D$ as
 follows.
Let $u$ and $v $ be the points of intersection of the line through the points $a,b$ with
the boundary $\partial D,$ ordered so that the points occur in the order $u,a,b,v$ on
the line. 
Then
the {\it Hilbert distance} $h_D(a,b)$ is defined by
\begin{equation} \label{hildef}
h_D(a,b) =  \log \left| u, a, b, v\right|
\end{equation}  
where
\begin{equation} \label{cr}
	\left| u, a, b, v\right|  = \frac{\left| u - b\right|  \left| a - v\right| }{\left| u - a\right|  \left| b - v\right| }  \,.
\end{equation} 
Hilbert extended the idea of the Cayley--Klein model of hyperbolic
geometry to arbitrary bounded convex domains in the plane. The Hilbert
metric in a bounded convex domain $D \subset \mathbb{R}^n$ is Finslerian
and, assuming that $D$ has a smooth strongly convex boundary, it is
Riemannian if and only if the boundary of the domain is an ellipsoidal
surface; see \cite[Theorem 11.6, p.~103]{pt}. For a study of geodesics, see Busemann \cite{bu}. This model of geometry has been recently studied  from various points of view, e.g. in
 \cite{p, sop,rv} and in computer graphics  \cite{ni}.  For the special case of the domain $D= \mathbb{B}^2,$ Hilbert's metric is also known as the Cayley-Klein metric. 
 
We continue here our work from \cite{afv} by proving a H\"older continuity
result for a quasiregular mapping of the unit disk $\mathbb{B}^2$
onto a bounded convex domain and giving
several formulas involving the Hilbert and other metrics.

After this introduction, the content of the paper is organized as follows. In Section 2 we give some notation and preliminary information and thereafter, 
in Section 3, we give several upper and lower bounds
for the Hilbert metric.
 In Section 4 we study Hilbert circles in polygonal plane domains and for dimensions $n\ge 3,$ give the formulas for Hilbert spheres in $\mathbb{B}^n,$ 
 which are known to be Euclidean ellipsoids, 
 extending the result from \cite{afv} for the case $n=2$ .
 In Section 5 one of the main results is given, we prove a H\"older continuity result
for a quasiregular mapping $f: \mathbb{B}^2 \to D$ onto
a convex domain in terms of the Hilbert metric in the following form.

\begin{thm} \label{qrHolder} For all distinct $a,b \in \mathbb{B}^2$, let $f: \mathbb{B}^2 \to f(\mathbb{B}^2) = D$ be a $K$-quasiregular mapping onto a convex bounded domain $D$. Then
	\begin{equation}\label{hm}
		h_D(f(a), f(b)) \leq \frac{2 \,c(K)}{\sqrt{1 - m^2}} \max \left\lbrace h_{\mathbb{B}^2}(a, b), (h_{\mathbb{B}^2}(a, b))^{1/K} \right\rbrace,
	\end{equation}
where $ m $ denotes the Euclidean distance from the origin to the line $ L[a, b] $, and $c(K), $ $c(K) \to 1$ when $K \to 1,$ is as given in Theorem \ref{2.12}.
\end{thm}

A few remarks clarify some special cases of this result. First, if $K=1$
$K$-quasiregular mappings are analytic functions of function theory. Second,  injective $K$-quasiregular mappings are $K$-quasiconformal mappings. We believe that Theorem \ref{qrHolder} is new
in both cases. Yet, there might be room
for improvement in Theorem \ref{qrHolder}, see Remark \ref{sharpness}.

It should be also noted that
in the course of this work, we formulate several open problems.

\section{Preliminary facts} \label{sec2}
%
In this section we give some basic notation and facts that will be utilized in the paper.

The $ n $-dimensional unit ball is denoted by $ \mathbb{B}^n $, while the unit sphere in $\mathbb{R}^n $ is $ \mathbb{S}^{n-1} $. A ball with center
$x$ and radius $r$ is $B^n(x,r)$ and its boundary sphere $S^{n-1}(x,r).$
For $ a \in \mathbb{R}^n \setminus \left\lbrace 0\right\rbrace  $, let 
\( a^* = \frac{a}{|a|^2} \). 
We identify $  \mathbb{R}^2 $ with the complex plane $ \mathbb{C} $
and for  $ a \in \mathbb{R}^2 $, the complex conjugate $ \overline{a} $ of $ a $  is $\overline{a} = \operatorname{Re}(a) - \operatorname{Im}(a)i.$


Assume that $ L[a, b] $ represents the line passing through points $ a $ and $ b $ ($ b \neq a $). For distinct points $ a, b, c, d \in \mathbb{C}$, if the lines $ L[a, b] $ and $ L[c, d] $ intersect at a single point $ w $, then we use the symbol $\text{LIS}[a, b, c, d]$ for this point of intersection

\begin{equation*}
w = \text{LIS}[a, b, c, d] = L[a, b] \cap L[c, d].
\end{equation*}
The formula for $w$ is  (see e.g., \cite[Ex. 4.3(1), p. 57 and p. 373]{hkv})

\begin{equation}\label{2.1}
	w = \text{LIS}[a, b, c, d] = \frac{( \overline{a}b -a\overline{b} )(c - d) - (a - b)(\overline{c}d -c \overline{d}  )}{(\overline{a} - \overline{b})(c - d) 
	-(a - b)(\overline{c} - \overline{d})}.
\end{equation}
 
For $a,b,c,d\in \partial \mathbb{B}^{2}$, using $a\overline{a}=b\overline{b}%
=c\overline{c}=d\overline{d}=1$, we obtain 
\begin{equation}
\mathrm{LIS}\left[ a,b,c,d\right] =\frac{ab\left( c+d\right) -cd\left(
a+b\right) }{ab-cd}=\frac{\overline{a}+\overline{b}-\overline{c}-\overline{d}%
}{\overline{a}\overline{b}-\overline{c}\overline{d}}.  \label{LIScircle}
\end{equation}

Let $ C[a, b, c]$ be the unique circle through three distinct, non-collinear points $ a $, $ b $, and $ c $. The formula (\ref{2.1}) easily yields a formula for the center and radius of this circle.

A M\"obius transformation is a mapping of the form
\begin{equation*}
	z \mapsto \frac{az + b}{cz + d}, \ \ \ a, b, c, d, z \in \mathbb{C}, \ ad - bc \neq 0.
\end{equation*}
The most important features of M\"obius transformations are that they  preserve 
the cross-ratio and the angle magnitude, and, because of this, they map every Euclidean line or circle onto either a line or a circle. The special M\"obius transformation
\begin{equation} \label{mb}
	T_a(z) = \frac{z - a}{1 - \overline{a}z}, \ \ \ a \in \mathbb{B}^2 \setminus \left\lbrace 0\right\rbrace 
\end{equation}
maps the unit disk $ \mathbb{B}^2 $ onto itself with
\begin{equation*}
	T_a(a) = 0, \ \ \ T_a\left(\pm \frac{a}{\left| a\right| }\right) = \pm \frac{a}{\left| a\right| }.
\end{equation*}


\begin{nonsec}{\bf Hyperbolic geometry.}\label{hg}
We review fundamental formulas and notation for hyperbolic geometry, for more information, see \cite{b}.

The hyperbolic metrics of the unit disk ${\mathbb{B}^2}$ and the upper half plane  ${\mathbb{H}^2}$ are given, respectively, by
\begin{equation}\label{rhoB}
\sh \frac{\rho_{\mathbb{B}^2}(a,b)}{2}=
\frac{\left| a-b\right| }{\sqrt{(1-\left| a\right| ^2)(1-\left| b\right| ^2)}} ,\ \ \ a,b\in \mathbb{B}^2,
\end{equation}
and
\begin{equation}\label{rhoH}
{\rm ch}\rho_{\mathbb{H}^2}(a,b)=1+ \frac{\left| a-b\right| ^2}{2 {\rm Im}(a) {\rm Im}(b)},\ \ \ a,b\in \mathbb{H}^2\,.
\end{equation}
Both metrics are M\"obius invariant: if $G, D \in \left\lbrace  \mathbb{B}^2, \mathbb{H}^2\right\rbrace $
and $f:G \to D= f(G)$ is a M\"obius transformation, then  $\rho_G(a,b)= \rho_D(f(a),f(b))$ holds for all $a,b \in G$. The geodesic line through
$a,b \in G$ is an arc of a circle, perpendicular to $\partial G\,.$ The end points of this arc on
$\partial G$ are $a_*$ and $b_*$ with $|a-a_*|\le |b-a_*|.$ In terms of these end points we have \cite[(4.9), (4.17)]{hkv}
\begin{equation} \label{endptrho}
\rho_G(a,b) = \log |a_*, a,b, b_*|\,.
\end{equation}
\end{nonsec}

We shall use the fact that a hyperbolic disk $ B_\rho(x,M)$ with the center $x \in {\mathbb{B}}^2$ and the radius $M>0$ is a Euclidean disk with the following
center and radius \cite[p. 56, (4.20)]{hkv}
\begin{equation}\label{hkv420}
 \begin{cases}
        B_\rho(x,M)=B^2(y,r)\;,&\\
  \noalign{\vskip5pt}
      {\displaystyle y=\frac{x(1-t^2)}{ 1-|x|^2t^2}\;,\;\;
        r=\frac{(1-|x|^2)t}{1-|x|^2t^2}\;,\;\;t={\rm th} ( M/2)\;.}&
\end{cases}
\end{equation}



Above ${\rm sh},$ ${\rm ch},$ and ${\rm th}$ stand for the hyperbolic sine, cosine, and tangent,  
and their inverse functions are ${\rm arsh},$
${\rm arch,}$ and ${\rm arth},$ respectively. We also use the notation
\[A[a,b]=    \sqrt{|a-b|^2 +(1-|a|^2)(1-|b|^2) } =  \sqrt{1+\left\vert a\right\vert ^{2}\left\vert b\right\vert
^{2}-2a\cdot b }\,.\]
Here the dot product of two points $ a, b \in \mathbb{R}^n $ is denoted by $ a \cdot b$.
For complex numbers $a,b$ we have $A[a,b]= | 1-\overline{a} b |.$

\begin{lem} \label{simpleidty}
The following  identity holds
for all $a,b \in \mathbb{B}^n$
\begin{equation} \label{SIDEN}
\left\vert a\left( 1-\left\vert b\right\vert ^{2}\right) +b\left(
1-\left\vert a\right\vert ^{2}\right) \right\vert ^{2}=\left( 1-\left\vert
a\right\vert ^{2}\left\vert b\right\vert ^{2}\right) ^{2}-\left(
1-\left\vert a\right\vert ^{2}\right) \left( 1-\left\vert b\right\vert
^{2}\right) A[a,b]^{2}. 
\end{equation}
\end{lem}

\begin{proof}
 Using
$$\left\vert a\left( 1-\left\vert b\right\vert ^{2}\right) +b\left(
1-\left\vert a\right\vert ^{2}\right) \right\vert ^{2}=\left\vert
a\right\vert ^{2}\left( 1-\left\vert b\right\vert ^{2}\right)
^{2}+\left\vert b\right\vert ^{2}\left( 1-\left\vert a\right\vert
^{2}\right) ^{2}+2a\cdot b\left( 1-\left\vert a\right\vert ^{2}\right)
\left( 1-\left\vert b\right\vert ^{2}\right) $$
and the elementary identity

\[\left\vert a\right\vert ^{2}\left( 1-\left\vert b\right\vert ^{2}\right)
^{2}+\left\vert b\right\vert ^{2}\left( 1-\left\vert a\right\vert
^{2}\right) ^{2}-\left( 1-\left\vert a\right\vert ^{2}\left\vert
b\right\vert ^{2}\right) ^{2}=-\left( 1-\left\vert a\right\vert ^{2}\right)
\left( 1-\left\vert b\right\vert ^{2}\right) \left( 1+\left\vert
a\right\vert ^{2}\left\vert b\right\vert ^{2}\right) ,\]
we get \[\left\vert a\left( 1-\left\vert b\right\vert ^{2}\right) +b\left(
1-\left\vert a\right\vert ^{2}\right) \right\vert ^{2}-\left( 1-\left\vert
a\right\vert ^{2}\left\vert b\right\vert ^{2}\right) ^{2}=\]
\[-\left(
1-\left\vert a\right\vert ^{2}\right) \left( 1-\left\vert b\right\vert
^{2}\right) \left( 1+\left\vert a\right\vert ^{2}\left\vert b\right\vert
^{2}-2a\cdot b\right) \,.\]
\end{proof}

\begin{nonsec}{\bf Tangents to hyperbolic geodesics.}
Suppose that $a,b \in \mathbb{B}^2 \setminus \{0\}$ are two points
on  the circle $S(c,\sqrt{|c|^2-1})$ orthogonal to the unit circle with $|a|\neq |b|$ and tangent lines are drawn to  $S(c,\sqrt{|c|^2-1})$ at the points $a$ and $b.$ First, we consider
the problem of finding the point of intersection of
these lines.

Simply, we use \eqref{2.1} to find the point 
\begin{equation} \label{eq:cen}
\text{LIS}[a,a+ i(a-c), b, b+ i (b-c)] = \frac{a(1-|b|^2)+b(1-|a|^2)}{2-a\overline{b}-\overline{a}b} \equiv {\rm cen}
\end{equation}
obtained by symbolic computation. By geometry, it is clear that the point $ cen$ is the Euclidean center of the hyperbolic circle through the points   $a$ and $b$ with the hyperbolic radius 
$\rho_{\mathbb{B}^2}(a,b)/2.$ In conclusion, recalling the formula
for the hyperbolic midpoint $m_{ab}$ of  $a$ and $b$ from \cite[Thm 1.4]{wvz} 
\begin{equation} \label{eq:z}
  m_{ab}   = \frac{a(1-|b|^2)+b(1-|a|^2)}
          {1-|a|^2|b|^2+\sqrt{(1-|a|^2)(1-|b|^2)A[a,b]^2}} \,,
\end{equation}
we have
\[
B^2(cen, |cen-a|)= B_{\rho}(m_{ab},\rho_{\mathbb{B}^2}(a,b)/2) \,.
\]
It is easily seen that the point $p=L[0,cen] \cap [a,b]$ is given by
\begin{equation}
\label{eq:p}
  p   = \frac{a(1-|b|^2)+b(1-|a|^2)}{2-|a|^2-|b|^2} \,.
\end{equation}

\begin{figure}[ht] %
\centerline{
\scalebox{0.7}	
		{\includegraphics{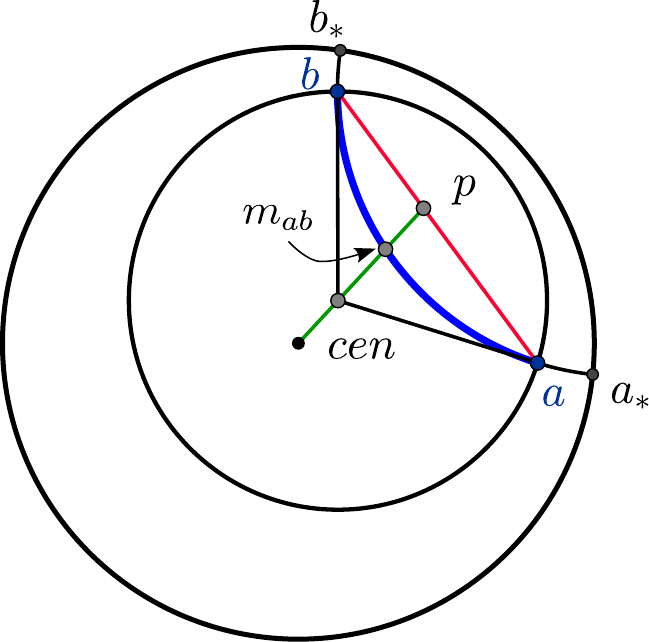}}	
	}
\caption{The points $a,b, {\rm cen},p, m_{ab} $ in Lemma \ref{lemFuji}.}
	\label{fig:abz}
\end{figure}
\bigskip

\bigskip
\end{nonsec}
%

\bigskip

We next apply \eqref{2.1} to prove the following Lemma \ref{lemFuji}.
\bigskip

\begin{lem} \label{lemFuji}
The point $m_{ab}$ is also
the hyperbolic midpoint of the segment $[cen,p],$ i.e.
$ \rho_{ \mathbb{B}^2}(cen,m_{ab})=\rho_{\mathbb{B}^2}(m_{ab},p) $ holds.
 Moreover, 
${B}^{2}(cen,\left\vert cen-a\right\vert )={B}_{\rho
}\left( m_{ab},\frac{\rho \left( a,b\right) }{2}\right) $.
\end{lem}
\begin{proof}
From \eqref{eq:cen} and \eqref{eq:p},
$$
   \arg cen=\arg p = \arg \big(a(1-|b|^2)+b(1-|a|^2)\big)
$$
hold.
Therefore, the intersection point of the line $ L[cen,p] $
and the hyperbolic line passing through $ a,b $ coincides with
$ m_{ab} $.

We will show that the hyperbolic midpoint of $ cen $  and  
$ p $ coincides with $ m_{ab} $.

Since the hyperbolic distance of 
$ a $ and  $ b $ can be written as 
$ \rho(a,b)=2 {\rm arth}\Big|\dfrac{a-b}{1-\overline{a}\,b}\Big| $,
we will find $ z $ that satisfies 
$ \Big|\dfrac{cen-z}{1-\overline{cen}\, z}\Big|
  =\Big|\dfrac{p-z}{1-\overline{p}\,z}\Big| $.
As this point $ z $ lies on $ L[cen,p] $,
it satisfies $ \arg z=\big(a(1-|b|^2)+b(1-|a|^2)\big) $.
So, we can set,
$$
   cen=\alpha C,\quad p=\alpha P,\quad z=\alpha Z ,
$$
where 
\begin{equation}\label{eq:CP}
    C=\frac{1}{2-\overline{a}b-a\overline{b}}, \ 
    P=\frac{1}{2-|a|^2-|b|^2}, \ \textrm{and} \  Z\in\mathbb{R}, 
\end{equation}
and $ \alpha=a(1-|b|^2)+b(1-|a|^2) $.
Then, we have
$$
    \Big|\dfrac{cen-z}{1-\overline{cen}\, z}\Big|
        =|\alpha|\Big|\dfrac{C-Z}{1-|\alpha|^2CZ}\Big| \quad \textrm{and}\quad
    \Big|\dfrac{p-z}{1-\overline{p}\,z}\Big|
        =|\alpha|\Big|\dfrac{P-Z}{1-|\alpha|^2PZ}\Big|.
$$
Considering the order of the three points $ cen,p,z $,
it suffices to find $Z$ such that 
\begin{equation}\label{eq:ZPC}
   \dfrac{Z-C}{1-|\alpha|^2CZ}=\dfrac{P-Z}{1-|\alpha|^2PZ}.
\end{equation}
Substituting \eqref{eq:CP} into \eqref{eq:ZPC}, we have
$$
  \frac{|a+b|^2-4}{(|a|^2+|b|^2-2)(a\overline{b}+\overline{a}b-2)}\, F=0,
$$
where
$$
  F=\big(a(1-|b|^2)+b(1-|a|^2)\big)
       \big(\overline{a}(1-|b|^2)+\overline{b}(1-|a|^2)\big)Z^2
      -2(1-|ab|^2)Z+1.
$$
The equation $ F=0 $ has the solution
$$
    Z=\frac{1-|ab|^2\pm|1-\overline{a}b|\sqrt{(1-|a|^2)(1-|b|^2)}}
           {\big(a(1-|b|^2)+b(1-|a|^2)\big)
            \big(\overline{a}(1-|b|^2)+\overline{b}(1-|a|^2)\big)}.
$$
Recalling that $ z=\alpha Z $ and considering the order of $ cen,p,z $ again,
we can check that $ z $ coincides with  $ m_{ab} $ 
using the identity (\ref{SIDEN}).

In order to prove that $ {B}_{\rho }\left(m_{ab},\frac{%
\rho \left( a,b\right) }{2}\right) = {B}^{2}\left( cen,\left\vert
cen-a\right\vert \right) $ it suffices to prove that 
$cen=\frac{m_{ab}\left(
1-t^{2}\right) }{1-\left\vert 
m_{ab}\right\vert ^{2}t^{2}}$ for $t=\th \frac{%
\rho \left( a,b\right) }{4}$, which implies $ {B}_{\rho }\left(m_{ab},%
\frac{\rho \left( a,b\right) }{2}\right) = {B}^{2}\left( cen,r\right) $
for some $r>0$. Since $a\in \partial  {B}_{\rho }
\left(m_{ab},\frac{\rho\left( a,b\right) }{2}\right) 
=\partial  {B}^{2}\left( cen,r\right) $,
it follows that $r=\left\vert cen-a\right\vert $.

As $\th \left( \frac{s}{2}\right) =\frac{\th (s)}{1+\sqrt{1-\th ^{2}(s)%
}}$ for every $s\in \mathbb{R}$, it follows that 
\[t=\frac{\left\vert
a-b\right\vert }{A\left[ a,b\right] +\sqrt{\left( 1-\left\vert a\right\vert
^{2}\right) \left( 1-\left\vert b\right\vert ^{2}\right) }}.\] 
Then 
\[t^{2}=
\frac{A\left[ a,b\right] ^{2}-\left( 1-\left\vert a\right\vert ^{2}\right)
\left( 1-\left\vert b\right\vert ^{2}\right) }{\left( A\left[ a,b\right] +%
\sqrt{\left( 1-\left\vert a\right\vert ^{2}\right) \left( 1-\left\vert
b\right\vert ^{2}\right) }\right) ^{2}}=\frac{A\left[ a,b\right] -\sqrt{%
\left( 1-\left\vert a\right\vert ^{2}\right) \left( 1-\left\vert
b\right\vert ^{2}\right) }}{A\left[ a,b\right] +\sqrt{\left( 1-\left\vert
a\right\vert ^{2}\right) \left( 1-\left\vert b\right\vert ^{2}\right) }} \, .\] 

Using identity (\ref{SIDEN}), we get 
\[
\left\vert m_{ab}
\right\vert ^{2}=\frac{1-\left\vert a\right\vert ^{2}\left\vert
b\right\vert ^{2}-A\left[ a,b\right] \sqrt{\left( 1-\left\vert a\right\vert
^{2}\right) \left( 1-\left\vert b\right\vert ^{2}\right) }}{1-\left\vert
a\right\vert ^{2}\left\vert b\right\vert ^{2}+A\left[ a,b\right] \sqrt{%
\left( 1-\left\vert a\right\vert ^{2}\right) \left( 1-\left\vert
b\right\vert ^{2}\right) }}\text{. } 
\]%
Denote $A=A\left[ a,b\right] $, $R=\sqrt{\left( 1-\left\vert a\right\vert
^{2}\right) \left( 1-\left\vert b\right\vert ^{2}\right) }$ and $%
v=1-\left\vert a\right\vert ^{2}\left\vert b\right\vert ^{2}$. Then 
\[
\frac{m_{ab}\left( 1-t^{2}\right) }{1-
\left\vert m_{ab}\right\vert ^{2}t^{2}}=\frac{%
a\left( 1-\left\vert b\right\vert ^{2}\right) +b\left( 1-\left\vert
a\right\vert ^{2}\right) }{1-\left\vert a\right\vert ^{2}\left\vert
b\right\vert ^{2}+A\left[ a,b\right] ^{2}}. 
\]%
But 
\[1-\left\vert a\right\vert ^{2}\left\vert b\right\vert ^{2}+A\left[ a,b%
\right] ^{2}=1-\left\vert a\right\vert ^{2}\left\vert b\right\vert
^{2}+\left\vert a-b\right\vert ^{2}+\left( 1-\left\vert a\right\vert
^{2}\right) \left( 1-\left\vert b\right\vert ^{2}\right) =2-\left( a%
\overline{b}+\overline{a}b\right) ,\] 
hence 
\[\frac{m_{ab}\left( 1-t^{2}\right) }{%
1-\left\vert m_{ab}
\right\vert ^{2}t^{2}}=\frac{a\left( 1-\left\vert b\right\vert
^{2}\right) +b\left( 1-\left\vert a\right\vert ^{2}\right) }{2-\left( a%
\overline{b}+\overline{a}b\right) }=cen,\] 
as claimed.
\end{proof}

\begin{nonsec}{\bf Remark.}
In the limiting case $|a|= |b|=1$
Lemma \ref{lemFuji} yields a well-known result, see \cite[Proposition 3.1]{vw2}.
\end{nonsec}
\section{Hilbert metric in the unit disk}

In this section we discuss the connection of the Hilbert metric
with the hyperbolic and the Euclidean metrics. We  first recall some well-known results and then prove lower and upper bounds for the Euclidean metric 
in terms of the Hilbert metric.  Note that the above metrics $\rho_{\mathbb{B}^2}, \rho_{\mathbb{H}^2},
$ and $h_{\mathbb{B}^2}$ can be used also in the one-dimensional case.

Recalling that ${\rm th}^2 (u/2)= \frac{{\rm ch}\, u-1}{{\rm ch}\, u +1}$ for $u>0$ and \cite[Cor 3.5]{rv}
\[
{\rm ch}(h_{\mathbb{B}^n}(a,b)/2)= \frac{1- a \cdot b}{\sqrt{\left( 1-\left\vert a\right\vert ^{2}\right) \left(
1-\left\vert b\right\vert ^{2}\right) }}
\]
we have
\begin{equation} 
{\rm th} ^{2}\left( \frac{h_{\mathbb{B}^{n}}\left( a,b\right) }{4}\right) =%
\frac{1-a\cdot b-\sqrt{\left( 1-\left\vert a\right\vert ^{2}\right) \left(
1-\left\vert b\right\vert ^{2}\right) }}{1-a\cdot b+\sqrt{\left(
1-\left\vert a\right\vert ^{2}\right) \left( 1-\left\vert b\right\vert
^{2}\right) }} \,. \label{explicit2}
\end{equation}%


The geodesic line through $a, b \in \mathbb{B}^2$ in the Hilbert metric
is a chord of the unit circle joining  the points $u,v \in \partial \mathbb{B}^2$ with
\begin{equation}
u= c+ \frac{(a-b)}{|a-b|} \sqrt{1- |c|^2}\, , \quad v= c+ \frac{(b-a)}{|a-b|} \sqrt{1- |c|^2}\, ,
\end{equation} 
and the Hilbert distance is $\log|u,a,b,v|$  
where $c={\rm LIS}[a,b,0, i(b-a)] \,.$ The points $u,a,b,v$ are on the chord $[u,v]$ in this order.

Unlike the hyperbolic metric $ \rho_{\mathbb{B}^n} $, the Hilbert metric is not invariant under the M\"obius automorphisms of $ \mathbb{B}^n $ as indicated by the following theorem.

\begin{thm} (See \cite[Thm 1.2]{rv}) \label{4.2}
For all distinct $a,b \in \mathbb{B}^2$, the following functional identity holds 
between the	Hilbert metric and the hyperbolic metric:
\begin{equation}\label{rveq}
		\sh \left( \frac{h_{\mathbb{B}^2}(a, b)}{2} \right) = 
		\sqrt{1 - m^2} \sh \left( \frac{\rho_{\mathbb{B}^2}(a, b)}{2} \right),
\end{equation}
where $ m $ is the Euclidean distance from the origin to the line $ L[a, b] $.
\end{thm}

\bigskip
%

\begin{prop} \label{FujiProp} For $t>0, c>0$ let
$f(t) = 2 \, {\rm arsh} (c \, {\rm sh}(\frac{t}{2}))$ and $g(t) = ct\,.$ Then for $c \in (0,1)$ the function
$f(t)/g(t): (0,\infty)\to (1, 1/c)$ is increasing and for $c>1$ the function $f(t)/g(t): (0,\infty)\to ( 1/c,1)$ is decreasing. In particular, for $c \in (0,1), t \ge 0$
we have
\[
 ct \le  2 \, {\rm arsh} (c \, {\rm sh}(\frac{t}{2})) \le t 
\]
and for $c >1, t \ge 0$
\[t \le  2 \, {\rm arsh} (c \, {\rm sh}(\frac{t}{2})) \le ct \,.
\]
\end{prop}

\begin{proof} Differentiation yields
	\[
	\frac{f'(t)}{g^\prime(t)} =  \frac{\ch(t/2)}{\sqrt{1 + c^2 \sh^2(t/2)}} = \quad \sqrt{\frac{1 + u^2}{1 + c^2 u^2}}; \quad u = \operatorname{sh}(t/2)
	\]
	and hence \( \frac{f'(t)}{g'(t)} \) is increasing for \( c \in (0,1) \) and decreasing for \( c >1 \). 
Now, from \cite[Thm B.2, p. 465]{hkv}, it
follows that also the function \(
\frac{f(t)}{g(t)} \) has the same monotonicity property. In particular, the conclusion holds.
\end{proof}

\begin{cor} \label{hrho} For all distinct $a,b \in \mathbb{B}^2$, the following inequality holds 
between the	Hilbert metric and the hyperbolic metric:
\begin{equation}\label{rhohil}
h_{\mathbb{B}^2}(a, b) \le \rho_{\mathbb{B}^2}(a, b) \le \frac{h_{\mathbb{B}^2}(a, b)}{\sqrt{1 - m^2}}
		\end{equation}
where $ m $ is the Euclidean distance from the origin to the line $ L[a, b] .$ Here equality holds if $m=0.$
\end{cor}

\begin{proof}
The proof follows from Theorem \ref{4.2} and Proposition \ref{FujiProp}.
\end{proof}
\bigskip

\begin{cor}
For all $a,b\in \mathbb{B}^{n}$, 
\begin{equation}
{\rm th} \left( \frac{h_{\mathbb{B}^{n}}\left( a,b\right) }{4}\right) \geq 
\frac{\left\vert a-b\right\vert }{\sqrt{4-\left\vert a+b\right\vert ^{2}}}. 
\label{Est2h}
\end{equation}
Equality holds if and only if
\begin{eqnarray*}
	\left| a\right| = \left| b \right|.
\end{eqnarray*}
Indeed, equality holds in particular when $a=-b.$
\end{cor}

\begin{proof}
The function $t\mapsto \frac{k-t}{k+t}$ with $k>0$ is decreasing for $t\in
\lbrack 0,\infty ).$

By the mean inequality $\sqrt{\left( 1-\left\vert a\right\vert ^{2}\right)
\left( 1-\left\vert b\right\vert ^{2}\right) }\leq 1-\frac{\left\vert
a\right\vert ^{2}+\left\vert b\right\vert ^{2}}{2}$. Then by \eqref{explicit2}

\[
{\rm th} ^{2}\left( \frac{h_{\mathbb{B}^{n}}\left( a,b\right) }{4}\right) =%
\frac{1-a\cdot b-\sqrt{\left( 1-\left\vert a\right\vert ^{2}\right) \left(
1-\left\vert b\right\vert ^{2}\right) }}{1-a\cdot b+\sqrt{\left(
1-\left\vert a\right\vert ^{2}\right) \left( 1-\left\vert b\right\vert
^{2}\right) }}\geq \frac{1-a\cdot b-\left( 1-\frac{\left\vert a\right\vert
^{2}+\left\vert b\right\vert ^{2}}{2}\right) }{1-a\cdot b+\left( 1-\frac{%
\left\vert a\right\vert ^{2}+\left\vert b\right\vert ^{2}}{2}\right) }. 
\]%
But $$1-a\cdot b-\left( 1-\frac{\left\vert a\right\vert ^{2}+\left\vert
b\right\vert ^{2}}{2}\right) =\frac{\left\vert a-b\right\vert ^{2}}{2}, \quad 
1-a\cdot b+\left( 1-\frac{\left\vert a\right\vert ^{2}+\left\vert
b\right\vert ^{2}}{2}\right) =2-\frac{\left\vert a+b\right\vert ^{2}}{2}\,,$$
hence (\ref{Est2h}) follows. The equality statement follows from the injectivity of the function
\[
t\mapsto \frac{kt}{k+t}, \qquad k>0,
\]
on $[0,1)$ and the fact that the mean inequality
\[
\sqrt{(1-|a|^2)(1-|b|^2)}
\leq
1-\frac{|a|^2+|b|^2}{2}
\]
becomes an equality if and only if
\[
1-|a|^2=1-|b|^2,
\]
i.e.,
\[
|a|=|b|.
\]
\end{proof}

\begin{cor} \label{mamo1} For $a,b \in {\mathbb{B}^n}$ we have
\begin{equation}
{\rm th} \left( \frac{\rho_{\mathbb{B}^{n}}\left( a,b\right) }{4}\right) \geq 
{\rm th} \left( \frac{h_{\mathbb{B}^{n}}\left( a,b\right) }{4}\right) \geq 
\frac{\left\vert a-b\right\vert }{\sqrt{4-\left\vert a+b\right\vert ^{2}}}\,,
\label{mamo2}
\end{equation}
with equality if $a=\pm b$.
\end{cor}

Note that Corollary \ref{mamo1} yields, in the case of the unit disk, a better bound than Proposition \ref{my210}. This raises the question whether 
Proposition \ref{my210} itself could be improved so that
 Corollary \ref{mamo1} would follow from it. In other words, is it true that for a bounded convex domain $D,$ the inequality $\rho_D\ge h_D \,$ holds?
This inequality, in turn, would follow from a positive solution of the open problem \ref{hilApo}.

\bigskip

	\begin{lem}\label{prop}
	Let $ -1 < a < b < 1 $. Then 
	\begin{equation*}
		\left| a - b\right| \leq 2 \, {\rm th} \frac{h_D(a,b)}{4},
	\end{equation*}
	where $D=(-1,1)\subset \mathbb{R}$. The inequality holds as an equality if and only if $a=-b\in (-1,0)$.
\end{lem} 	
\begin{proof} 
Since $-1<a<b<1$, $\left\vert a-b\right\vert =b-a$ and $h_{D}(a,b)=\log
|-1,a,b,1|=\log \frac{(1+b)(1-a)}{(1-b)(1+a)}$. Let $f(a,b)=b-a$ and $g(a,b)=%
\frac{(1+b)(1-a)}{(1-b)(1+a)}$. Note that $-1<a<b<1$ implies $g(a,b)>1$. 
\newline
Let $t>1$ be fixed. If $g(a,b)=t$, with $-1<a<b<1$, then%
\begin{equation*}
	th\frac{h_{D}(a,b)}{4}=th\left( \frac{1}{4}\log t\right) =\frac{\sqrt{t}-1}{%
		\sqrt{t}+1}.
\end{equation*}%
The task is to find the maximum of $f(a,b)$ subject to the constraint $%
g(a,b)=t$. \newline
The arc $\Gamma _{t}$ given by $g(a,b)=t$, with $-1<a<b<1$, has the implicit
equation 
\begin{equation*}
	(t-1)ab-(t+1)a+(t+1)b-(t-1)=0.
\end{equation*}%
Then $\Gamma _{t}$ is an open arc of hyperbola, with endpoints $(\pm 1,\pm
1) $. We extend $\Gamma _{t}$ to the compact arc $\overline{\Gamma _{t}^{{}}}%
=\Gamma _{t}\cup \left\{ (\pm 1,\pm 1)\right\} $, given by the above
equation, but with $-1\leq a\leq b\leq 1$. Being continuous on $\mathbb{R}%
^{2}$, the function $f$ attains a minimum and a maximum on the compact set $%
\overline{\Gamma _{t}^{{}}}$.

We use the method of Lagrange multipliers. The Lagrange function of our problem is
	\begin{equation*}
		L(a, b, \lambda) = f(a,b) - \lambda g(a,b) = b - a - \lambda \left( \frac{(1+b)(1-a)}{(1-b)(1+a)} \right).
	\end{equation*}
	For every every extremum point $(a,b)$ of $f(a,b)$ subject to the constraint 
	$g(a,b)=t$ there exists $\lambda \in \mathbb{R}$ such that $(a,b,\lambda )$
	is a solution of the following system of equations:
	\begin{equation}\label{1}
		\frac{\partial L}{\partial a}(a,b,\lambda )=-1+\lambda \frac{2(1+b)}{%
			(1-b)(1+a)^{2}}=0,
	\end{equation}
	
	\begin{equation}\label{2}
		\frac{\partial L}{\partial b}(a,b,\lambda )=1+\lambda \frac{2(a-1)}{%
			(1-b)^{2}(1+a)}=0,
	\end{equation}
	and
	\begin{equation}\label{3}
	g(a,b)=\frac{(1+b)(1-a)}{(1-b)(1+a)}=t.
	\end{equation}
	
	From (\ref{1}), $\lambda = \frac{(1-b)(1+a)^2}{2(b+1)}$. Substituting this value of $ \lambda $ into (\ref{2}), we obtain $1-\frac{(1-a)(1+a)}{(1-b)(1+b)}=0$, that
	is, $a=b$ or $a=-b$.
	
	If $a=b\in \left\{ \pm 1\right\} $, then $f(a,b)=0$. In the case $a=-b$, if $%
	(a,b,\lambda )$ is a solution of the system, then $a<0<b=-a$ and $\lambda =%
	\frac{(1+a)^{3}}{2(1-a)}$, hence $g(a,b)=g(a,-a)=\left( \frac{1-a}{1+a}%
	\right) ^{2}=t$. It follows that $a=-\frac{\sqrt{t}-1}{\sqrt{t}+1}$. Denote $%
	a_{0}=-\frac{\sqrt{t}-1}{\sqrt{t}+1}$ and $\lambda _{0}=\frac{(1+a_{0})^{3}}{%
		2(1-a_{0})}$.
	
	Since $f(\pm 1,\pm 1)=0<2\frac{\sqrt{t}-1}{\sqrt{t}+1}=f\left(
	a_{0},-a_{0}\right) $, it follows that $(a_{0},-a_{0})$ is the (unique)
	point of maximum for $f(a,b)$ subject to the constraint $g(a,b)=t>1$, with $%
	-1<a<b<1$.
	
	Then, for all $(a,b)\in \mathbb{R}^{2}$ with $-1<a<b<1$ satisfying $%
	g(a,b)=t>1$, we have 
	\begin{equation*}
		f(a,b)\leq f\left( a_{0},-a_{0}\right) =2\frac{\sqrt{t}-1}{\sqrt{t}+1}=2th%
		\frac{h_{D}(a,b)}{4},
	\end{equation*}%
	hence 
	\begin{equation*}
		\left\vert a-b\right\vert \leq 2th\frac{h_{D}(a,b)}{4}\text{,}
	\end{equation*}%
	as required. The equality holds if and only if $a=a_{0}$ and $b=-a_{0}$.
	
	Using an alternative method, from $g(a,b)=t>1$ with $a\in (-1,1)$ we obtain $%
	b=\frac{t(1+a)-(1-a)}{t(1+a)+(1-a)}:=b(a)\in (-1,1)$, hence $b(a)-a=\frac{%
		(1-a^{2})(t-1)}{t(1+a)+(1-a)}>0$. We find the extrema of $f(a,b(a))=b(a)-a$
	for $a\in (-1,1)$. But 
	\begin{equation*}
		\frac{\partial }{\partial a}\left( b(a)-a\right) =-\frac{t-1}{\left(
			t(1+a)+(1-a)\right) ^{2}}\left[ (t-1)a^{2}+2(t+1)a+(t-1)\right]
	\end{equation*}%
	and the roots of $(t-1)a^{2}+2(t+1)a+(t-1)$ are $a_{0}=-\frac{\sqrt{t}-1}{%
		\sqrt{t}+1}\in (-1,0)$ and $a_{1}=-1-2\frac{\sqrt{t}+1}{t-1}<-1$.
	
	As $a\in (-1,1)$, the function $a\mapsto f(a,b(a))$ is increasing on $%
	(-1,a_{0}]$ and decreasing on $[a_{0},1)$, therefore it attains its maximum
	only for $a=a_{0}=-\frac{\sqrt{t}-1}{\sqrt{t}+1}$ and its maximum is $%
	b(a_{0})-a_{0}=2\frac{\sqrt{t}-1}{\sqrt{t}+1}$.

\end{proof}

\bigskip

\begin{prop} \label{my210} Let $D\subset \mathbb{R}^2 $ be a bounded convex domain and $a,b \in D. $
Then

\[ \left| a-b \right| \le  \operatorname{diam}(D)
{\rm th}\left(\frac{h_D(a,b)}{2}\right).
\]
In the special case $D= \mathbb{B}^2$, the equality holds if and only if $a= b$ or $a= -b$.
\end{prop}

\begin{proof}
	If $a=b$, the inequality holds as an equality. 
	
	Now assume that $a$ and $b$ are distinct. The line $L[a,b]$ intersect the boundary of $D$ at $u$ and
$v$, where $u,a,b,v$ are in this order.

Using an affine transformation, we map the segment $[u,v]$ onto the segment $[-e_1,e_1]$, with $e_1=(1,0,\ldots,0)\in\mathbb{R}^n$, then we identify $S:=(-e_1,e_1)$ with $(-1,1)\subset\mathbb{R}$. There exists an affine transformation $T:\mathbb{R}^n\to\mathbb{R}^n$ such that $T(v)=e_1$ and $T(u)=-e_1$, defined by $T(x)=kA(x)+c$, where $A:\mathbb{R}^n\to\mathbb{R}^n$ is an orthogonal transformation, $k\in\mathbb{R}_+$ and $c\in\mathbb{R}^n$. From $2=|T(u)-T(v)|=k|A(u-v)|=k|u-v|$, we get $k=\frac{2}{|u-v|}$. 

Using Lemma \ref{prop}, it follows that $|T(a)-T(b)|\leq 2{\rm th}\frac{h_S(T(a),T(b))}{2}$. But the Hilbert distance is invariant under affine rescaling, therefore $h_S(T(a),T(b))=h_D(a,b)$ and $|T(a)-T(b)|=k|A(a-b)|=k|a-b|$. Then $\frac{2}{|u-v|}|a-b|\leq 2{\rm th}\frac{h_D(a,b)}{2}$, hence 
\begin{eqnarray*}
|a-b|\leq |u-v|{\rm th}\frac{h_D(a,b)}{2} \leq \operatorname{diam}(D){\rm th}\frac{h_D(a,b)}{2}.
\end{eqnarray*}
\end{proof}

The part (1) from the following lemma is well-known from projective geometry 
\cite[Section 4]{kob}, however we include a short proof and provide some
explicit formulas.

\begin{lem}\label{lem:cross-ratio}
	\label{Masayo_Kobayashi}Let $u,c,d,v,w$ be points on the unit circle $%
	\partial \mathbb{B}^{2}$, in this order. Let $a=\mathrm{LIS}\left[ u,v,c,w%
	\right] $ and $b=\mathrm{LIS}\left[ u,v,d,w\right] $. Then
	
	(1) The following equality of cross-ratios holds%
	\[
	\left\vert u,a,b,v\right\vert =\left\vert u,c,d,v\right\vert . 
	\]
	
	(2) The arc of the unit circle with endpoints $u$ and $v$ that contains $c,d$
	is mapped onto the segment $\left[ u,v\right] $ by the function defined by $F(z)=\mathrm{LIS}[u,v,z,w]$, which is a restriction of a M\"{o}bius
	transformation, namely 
	\[
	F(z)=\frac{\left( uv-uw-vw\right) z+uvw}{-wz+uv%
	}. 
	\]%
	In particular, $F(u)=u$, $F(v)=v$ and, by the definitions of $a$ and $b$, 
	\[
	a=F(c)\text{ and }b=F(d)\text{.} 
	\]
	
	(3) If the lines $L\left[ u,v\right] $ and $L\left[ c,d\right] $ are
	parallel, then $uv=cd$ and 
	\[
	a=\frac{(u+v-d)w-cd}{w-d}\,\,\text{ and }\quad b=\frac{(u+v-c)w-cd}{w-c}. 
	\]
\end{lem}

\begin{proof}
	(1) Denote $\alpha =\measuredangle (u,w,c)$, $\beta =\measuredangle \left(
	c,w,d\right) $ and $\gamma =\measuredangle (d,w,v)$.\newline
	We prove that 
	\begin{equation}
		\left\vert u,a,b,v\right\vert =\frac{\sin (\alpha +\beta )\sin (\beta
			+\gamma )}{\sin \alpha \sin \gamma }=\left\vert u,c,d,v\right\vert .
		\label{equal_CRs}
	\end{equation}%
	Using the law of sines in $\Delta( w,a,u)$ and $\Delta( w,b,u)$ and, similarly, in 
	$\Delta( w,a,v)$ and $\Delta( w,b,v),$
	we get
	\[
	\frac{\left\vert u-b\right\vert }{\left\vert u-a\right\vert }=\frac{
		\left\vert w-b\right\vert }{\left\vert w-a\right\vert }\frac{\sin \left(
		\alpha +\beta \right) }{\sin \alpha }\,, \quad \frac{\left\vert v-a\right\vert }{%
		\left\vert v-b\right\vert }=\frac{\left\vert w-a\right\vert }{\left\vert
		w-b\right\vert }\frac{\sin \left( \beta +\gamma \right) }{\sin \gamma }\,.\]
	Then 
	\[
	\left\vert u,a,b,v\right\vert =\frac{\left\vert u-b\right\vert }{\left\vert
		u-a\right\vert }\frac{\left\vert v-a\right\vert }{\left\vert v-b\right\vert }%
	=\frac{\sin (\alpha +\beta )\sin (\beta +\gamma )}{\sin \alpha \sin \gamma } 
	\]%
	and the first equality in (\ref{equal_CRs}) follows. 
	On the other hand, using the law of sines in the triangles 
	$\Delta( w,u,c)$, $
	\Delta (w,u,d),$ $\Delta( w,c,v)$, and $\Delta( w,d,v)$, 
	inscribed in the unit circle, we
	get \[\frac{\left\vert u-c\right\vert }{\sin \alpha }=\frac{\left\vert
		u-d\right\vert }{\sin \left( \alpha +\beta \right) }=\frac{\left\vert
		v-c\right\vert }{\sin \left( \beta +\gamma \right) }=\frac{\left\vert
		v-d\right\vert }{\sin \gamma }=2\,,\] 
	hence 
	\[
	\left\vert u,c,d,v\right\vert =\frac{\left\vert u-d\right\vert }{\left\vert
		u-c\right\vert }\frac{\left\vert v-c\right\vert }{\left\vert v-d\right\vert }%
	=\frac{\sin (\alpha +\beta )\sin (\beta +\gamma )}{\sin \alpha \sin \gamma }%
	\text{,} 
	\]%
	therefore the first equality in (\ref{equal_CRs}) is obtained. 
	\begin{figure}[H] 
		\centerline{\includegraphics[width=0.45\linewidth]{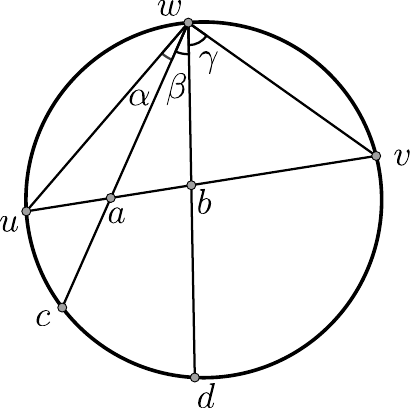}}
		\caption{Lemma \ref{lem:cross-ratio} (1).}
		\label{lem:(1)}
	\end{figure}
	
	(2) Let $F(z)=\mathrm{LIS}[u,v,z,w]$, where $z$ moves on the arc of the unit
	circle with endpoints $u$ and $v$ that contains $c,d$. Using (\ref{LIScircle}%
	), we get 
	\[
	F(z)=\frac{uv(z+w)-zw(u+v)}{uv-zw}=\frac{\left( uv-uw-vw\right) z+uvw}{-wz+uv%
	}. 
	\]%
	Note that $F(u)= u, F(v)= v, F(c)=a, F(d)=b $ and $F$ preserves the complex cross-ratio, since $F$ is a restriction of a M\"{o}bius transformation, therefore 
	\begin{eqnarray*}
		\left\vert u,a,b,v\right\vert 
		&=&\frac{\left\vert u-b\right\vert }{\left\vert u-a\right\vert }
		\frac{\left\vert v-a\right\vert }{\left\vert v-b\right\vert }
		=  \frac{|F\left( u\right) -F\left( d\right)|}
		{|F\left( u\right) -F\left( c\right)|}
		\frac{|F\left( v\right) -F(c)|}{|F(v)-F\left(d \right) |} \\
		&=& \frac{\vert u-d\vert}{\vert u-c\vert}
		\frac{\vert v-c\vert}{\vert v-d\vert}
		= \left\vert u,c,d,v\right\vert .
	\end{eqnarray*}%
	This provides another proof for the equality of cross-ratios from (1). 
	\newline
	(3) If $L\left[ u,v\right] $ and $L\left[ c,d\right] $ are parallel, then
	there exists $\theta \in \mathbb{R}$ such that $u=ce^{-i\theta }$ and $%
	v=de^{i\theta }$. Then $uv=cd$ and 
	\[
	a=F(c)=\frac{uv(c+w)-cw(u+v)}{uv-cw}=\frac{cd\left( c+w\right) -cw\left(
		u+v\right) }{cd-cw}=\frac{w\left( u+v-d\right) -cd}{w-d}. 
	\]%
	Similarly, $b=F(d)=\frac{w\left( u+v-c\right) -cd}{w-c}$.
\end{proof}

\begin{lem} \label{newLem}
	For distinct points $a,b\in \mathbb{B}^{2}$, let $u,v\in
	L[a,b]\cap \partial \mathbb{B}^{2}$, where $u,a,b,v$ are aligned in this
	order. Then, the following two statements hold.

	\begin{enumerate}
		\item There exist two circles passing through points $a$ and $b$, tangent to
		the unit circle, and the points of tangency lie in distinct half-planes
		determined by the line $L[a,b]$.
		
		\item\label{lem:2}
		Let $w\in \partial \mathbb{B}^{2}\setminus \left\{ u,v\right\} $ and $%
		c\in L[w,a]\cap \partial \mathbb{B}^{2}\setminus \left\{ w\right\} $, $%
		d\in L[w,b]\cap \partial \mathbb{B}^{2}\setminus \left\{ w\right\} $.
		The circle passing through $a$, $b$ and $w$ is tangent to the unit circle if
		and only if the lines $L[a,b]$ and $L[c,d]$ are parallel.
	\end{enumerate}
\end{lem}

\begin{proof}
	(1) Let $w\in \partial \mathbb{B}^{2}\setminus \left\{ u,v\right\} $ and $%
	C\left( a,b,w\right) $ be the circle passing through $a$, $b$ and $w$. Applying the
	cross-ratio we see that the equation of $C\left( a,b,w\right) $ is 
	\begin{equation} \label{abw}
		\left( a-w\right) \left( \overline{b}-\overline{w}\right) \left( b-z\right)
		\left( \overline{a}-\overline{z}\right) -\left( b-w\right) \left( \overline{a%
		}-\overline{w}\right) \left( a-z\right) \left( \overline{b}-\overline{z}%
		\right) =0. 
	\end{equation}%
	The unit circle and the circle $C\left( a,b,w\right) $ have the intersection
	points $w$ and $z_{0}$. Since $w$, $z_{0}\in \partial \mathbb{B}^{2}$,
	substituting in (\ref{abw}) $\overline{w}=1/w$ and $\overline{z}=\overline{%
		z_{0}}=1/z_{0}$, we obtain 
	\begin{eqnarray}
		&&\left( a-w\right) \left( \overline{b}w-1\right) \left( b-z_{0}\right)
		\left( \overline{a}z_{0}-1\right) -\left( b-w\right) \left( \overline{a}%
		w-1\right) \left( a-z_{0}\right) \left( \overline{b}z_{0}-1\right)  
		\nonumber \\
		&=&\left( w-z_{0}\right) \left( kz_{0}w-\left( \left\vert a\right\vert
		^{2}-\left\vert b\right\vert ^{2}\right) \left( z_{0}+w\right) +\overline{k}%
		\right) =0\text{,}  \label{abw2}
	\end{eqnarray}%
	where $k:=\left( \left( a-b\right) \overline{a}\overline{b}+\overline{a}-%
	\overline{b}\right) $. Note that $k=\left( 1-\left\vert b\right\vert
	^{2}\right) \overline{a}-\left( 1-\left\vert a\right\vert ^{2}\right) 
	\overline{b}$, and $k=0$ if and only if $a=b$. 
	The unit circle and $C\left( a,b,w\right) $ are tangent if and only if the
	equation (\ref{abw2}) has a double root $w=z_{0}$, i.e. $w$ is a root of the
	quadratic equation 
	\begin{equation} \label{abw3}
		\left( \left( a-b\right) \overline{a}\overline{b}+\overline{a}-\overline{b}%
		\right) w^{2}-2\left( \left\vert a\right\vert ^{2}-\left\vert b\right\vert
		^{2}\right) w+\left( \left( \overline{a}-\overline{b}\right) ab+a-b\right)
		=0. 
	\end{equation}%
	
	As $w\overline{w}=1$, (\ref{abw3}) holds if and only if $z=w$ satisfies 
	\begin{equation}\label{abwwb}
		\left( \left( a-b\right) \overline{a}\overline{b}+\overline{a}-\overline{b}%
		\right) z+\left( \left( \overline{a}-\overline{b}\right) ab+a-b\right) 
		\overline{z}-2\left( \left\vert a\right\vert ^{2}-\left\vert b\right\vert
		^{2}\right) =0.  
	\end{equation}%
	Equation (\ref{abwwb}) represents a line passing through the point $%
	(\left\vert a\right\vert ^{2}-\left\vert b\right\vert ^{2})/k$ and this
	point belongs to the unit disk, as will follows from the identity (\ref%
	{SIDEN2}) below. Therefore, the line (\ref{abwwb}) and the unit circle have
	exactly two intersection points $w_{1}\neq w_{2}$. The line (\ref{abwwb})
	intersects $L\left[ a,b\right] $ at a single point, therefore $w_{1}$ and $%
	w_{2}$ lie in distinct half-planes determined by $L\left[ a,b\right] $.
	Indeed, the lines (\ref{abwwb}) and 
	\[ L[a,b]: \left( -i \right) \left( 
	\overline{a}-\overline{b}\right) z+i\left( a-b\right) z+2 {\rm Im}(a%
	\overline{b})=0 \] 
	would be parallel or coincident if and only if ${\rm Re}%
	\left( k\left( a-b\right) \right) =0$, but 
	\begin{eqnarray*}
		{\rm Re}\left( k\left( a-b\right) \right)  &=&\left( \left\vert
		a\right\vert ^{2}+\left\vert b\right\vert ^{2}\right) \left( 1+{\rm Re}(a%
		\overline{b})\right) -2\left( 1+\left\vert a\right\vert ^{2}\left\vert
		b\right\vert ^{2}\right)  \\
		&<&-2\left( 1-\left\vert a\right\vert ^{2}\right) \left( 1-\left\vert
		b\right\vert ^{2}\right) <0.
	\end{eqnarray*}%
	Solving (\ref{abw3}) we obtain 
	\begin{equation}\label{tangency}
		\left\{ w_{1},w_{2}\right\} =\left\{ \frac{\left\vert a\right\vert
			^{2}-\left\vert b\right\vert ^{2}\pm i\left\vert a-b\right\vert \sqrt{\left(
				1-\left\vert a\right\vert ^{2}\right) \left( 1-\left\vert b\right\vert
				^{2}\right) }}{\left( a-b\right) \overline{a}\overline{b}+\overline{a}-%
			\overline{b}}\right\} .  
	\end{equation}%
	It is easy to check that $w_{1},w_{2}\in \partial \mathbb{B}^{2}$, observing
	that 
	\begin{equation}\label{SIDEN2}
		\left\vert \left( a-b\right) \overline{a}\overline{b}+\overline{a}-\overline{%
			b}\right\vert ^{2}=\left( \left\vert a\right\vert ^{2}-\left\vert
		b\right\vert ^{2}\right) ^{2}+\left\vert a-b\right\vert ^{2}\left(
		1-\left\vert a\right\vert ^{2}\right) \left( 1-\left\vert b\right\vert
		^{2}\right) \text{,}  
	\end{equation}
	as follows from (\ref{SIDEN}) for $n=2$ by substituting $a$ and $b$ by $%
	\overline{a}$ and $\left( -\overline{b}\right) $, respectively. 
	
	In conclusion, the circle $C\left( a,b,w\right) $ is tangent to the unit
	circle if and only if $w\in \left\{ w_{1},w_{2}\right\} $.
	
	(2) The circles $\partial \mathbb{B}^{2}$ and $C\left( a,b,w\right) $ are
	tangent at $w\in \partial \mathbb{B}^{2}$ if and only if $T_{1}=T_{2}$,
	where $T_{1}$ and $T_{2}$ are the tangent lines at $w\in \partial \mathbb{B}%
	^{2}$ to $\partial \mathbb{B}^{2}$ and $C\left( a,b,w\right) $,
	respectively. The angles formed by $T_{1}$ and $T_{2}$ with the line $%
	L[c,w]=L[a,w]$, in the half plane determined by this line and containing $u$, are equal to $\measuredangle (w,d,c)$ and $\measuredangle (w,b,a)$,
	respectively. Then the tangency of the circles $\partial \mathbb{B}^{2}$ and 
	$C\left( a,b,w\right) $ at $w\in \partial \mathbb{B}^{2}$ is equivalent to $%
	\measuredangle (w,d,c)=\measuredangle (w,b,a)$, i.e. to the parallelism of
	the lines $L[a,b]=L[u,v]$ and $L[c,d].$
	
\end{proof}

\begin{figure}[htbp]
	\centerline{\includegraphics[width=0.45\linewidth]{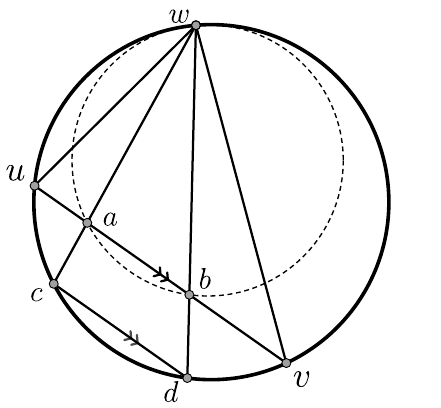}}
	\caption{Lemma \ref{newLem}(2): $ L[a,b] $ and $ L[c,d] $ 
		are parallel.}
	\label{fig:2-2}
\end{figure}

\begin{rem}
	From Lemma \ref{newLem}, we can see that, if $w\in \partial \mathbb{B}^{2}\setminus \left\{ u,v\right\} $ is chosen such that the
	circle passing through $a$, $b$ and $w$ is tangent to the unit circle, then the lengths of arcs 
	$ \stackrel{\large\mbox{$\frown$}}{uc} $ and 
	$ \stackrel{\large\mbox{$\frown$}}{dv} $ are equal and $
	L[a,b]$ is mapped to $L[c,d]$ through a similarity transformation
	with the form $f(z)=w+A\left( z-w\right) $, where $A\in \mathbb{R}$.
\end{rem}

\begin{thm}\label{Hilbertmidpoint} 
	\label{lemBisect} (1) Let $u,c,d,v,w$ be distinct points on the unit circle $\partial 
	\mathbb{B}^{2}$, in this order. Let $a=\mathrm{LIS}\left[ u,v,c,w\right] $
	and $b=\mathrm{LIS}\left[ u,v,d,w\right] $. Then 
	\begin{equation}
		h_{\mathbb{B}^{2}}(a,b)=\log |u,c,d,v|.  \label{Hcircle}
	\end{equation}
	
	(2) Assume that the lines $L[u,v]$ and $L[c,d]$ are parallel. Let $%
	m=(c+d)/|c+d|$ with $c\neq -d$ and $p=\mathrm{LIS}[u,v,w,m]$. Then the circle $C(w,a,b)$ is
	tangent to $\partial \mathbb{B}^{2}$ at the point $w$ and $\measuredangle
	(c,w,m)=\measuredangle (d,w,m)$. Moreover, $p$ is the Hilbert midpoint of $a$
	and $b$, hence $p$ is on the hyperbolic perpendicular bisector of the
	hyperbolic geodesic segment determined by $a$ and $b$,%
	\[
	h_{\mathbb{B}^{2}}(a,p)=h_{\mathbb{B}^{2}}(p,b)\,\quad \mathrm{and}\quad
	\rho _{\mathbb{B}^{2}}(a,p)=\rho _{\mathbb{B}^{2}}(p,b). 
	\]
\end{thm}

\begin{figure}[H]
	\centering
	\begin{subfigure}[b]{0.45\textwidth}
		\centering
		\includegraphics[width=\textwidth]{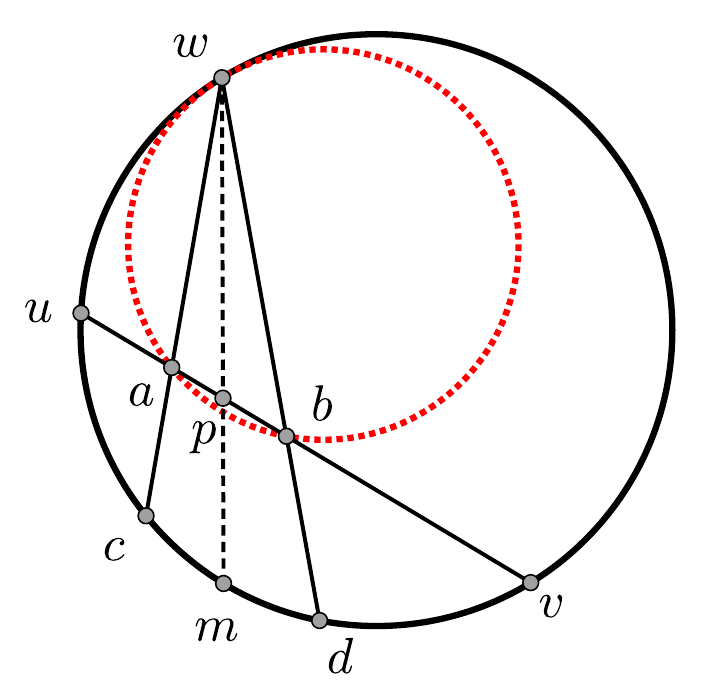}
		\caption{}
	\end{subfigure}
	\hfill
	\begin{subfigure}[b]{0.48\textwidth}
		\centering
		\includegraphics[width=\textwidth]{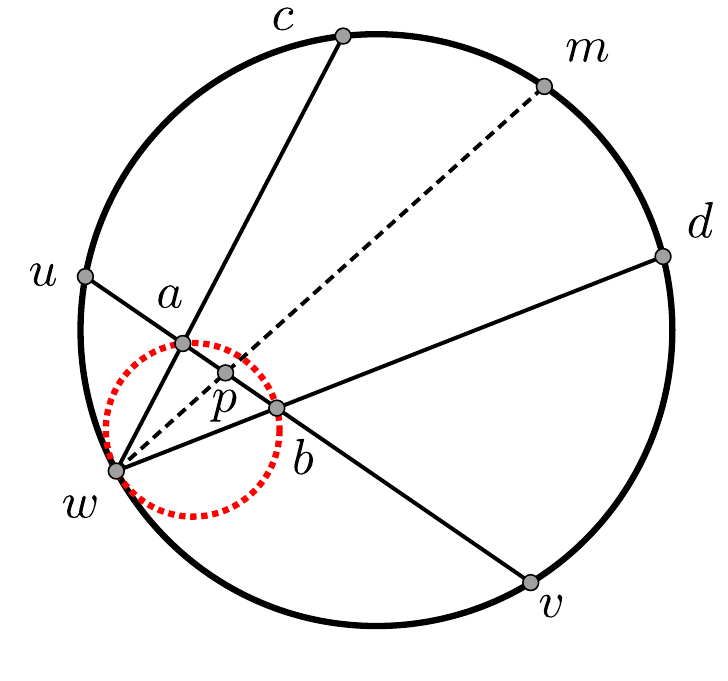}
		\caption{}
		\label{fig:2}
	\end{subfigure}
	\label{fig:3}
	\caption{The point $w$ in Theorem  \ref{lemBisect} (2) is chosen
		so that the circle $C(a, b,w)$ through the points $a, b,w,$
		is tangent to the unit circle, see Remark \ref{vamrmk}. In both cases (A) and (B) 
		$|c-u|=|d-v|$ and $|c-m|=|d-m|$ where $m=(c+d)/|c+d|\,.$}
	\label{fig:UVW}
\end{figure}

\begin{proof}
	(1) By definition, $h_{\mathbb{B}^{2}}(a,b)=\log \left\vert
	u,a,b,v\right\vert $. By Lemma \ref{lem:cross-ratio}, $\left\vert
	u,a,b,v\right\vert =\left\vert u,c,d,v\right\vert $ and (\ref{Hcircle})
	follows.
	
	(2) The point $m$ is the Euclidean midpoint of the arc determined on the
	unit circle by $c$ and $d$, in the half plane not containing $u,v$, $w$.
	Then $\measuredangle (c,w,m)=\measuredangle (d,w,m)$ and $p$ is on the
	segment $\left[ a,b\right] $. As $L[u,v]$ and $L[c,d]$ are parallel, it
	follows that $\alpha =\gamma $ and $\left\vert m-c\right\vert =\left\vert
	m-d\right\vert $, $\left\vert u-c\right\vert =\left\vert v-d\right\vert $,
	and Lemma \ref{newLem} (2) \ shows that $C\left( w,a,b\right) $ is tangent
	to the unit circle at $w$.
	
	Using equalities analogous to (\ref{equal_CRs}), from the proof of Lemma \ref
	{lem:cross-ratio}, we obtain 
	\[
	h_{\mathbb{B}^{2}}(a,p)=\log |u,a,p,v|=\log |u,c,m,v|=\log \frac{\sin \left(
		\alpha +\frac{\beta }{2}\right) \sin \left( \beta +\gamma \right) }{\sin
		\alpha \sin \left( \gamma +\frac{\beta }{2}\right) }
	\]%
	and 
	\[
	h_{\mathbb{B}^{2}}(p,b)=\log |u,p,b,v|=\log |u,m,d,v|=\log \frac{\sin \left(
		\alpha +\beta \right) \sin \left( \gamma +\frac{\beta }{2}\right) }{\sin
		\left( \alpha +\frac{\beta }{2}\right) \sin \gamma }.
	\]%
	Since $\alpha =\gamma $, it follows that $h_{\mathbb{B}^{2}}(a,p)=h_{\mathbb{%
			B}^{2}}(p,b)$. Alternatively, we may use in the proof the formulas 
	\[
	|u,c,m,v|=\frac{\left\vert u-m\right\vert }{\left\vert u-c\right\vert }\frac{
		\left\vert v-c\right\vert }{\left\vert v-m\right\vert }, \quad |u,m,d,v|=\frac{
		\left\vert u-d\right\vert }{\left\vert u-m\right\vert }\frac{\left\vert
		v-m\right\vert }{\left\vert v-d\right\vert }\] 
	and the equalities $\left\vert
	v-c\right\vert =\left\vert u-d\right\vert $, $\left\vert u-c\right\vert
	=\left\vert v-d\right\vert $, and $\left\vert u-m\right\vert =\left\vert
	v-m\right\vert $.
	
	It is known that $h_{\mathbb{B}^{2}}(a,p)=h_{\mathbb{B}^{2}}(p,b)$, where $p$
	is on the segment $\left[ a,b\right] $, implies $\rho _{\mathbb{B}%
		^{2}}(a,p)=\rho _{\mathbb{B}^{2}}(p,b)$, see \cite[Theorem 6.21]{afv} or 
	\cite[Theorem 1.2]{rv}. 
\end{proof}

\begin{nonsec}{\bf Remark.}\label{vamrmk} 
	In \cite[Theorem 1.1]{fkv} the visual angle distance in the unit disk
	between two points $a,b \in \mathbb{B}^2$, with $\left\vert
	a\right\vert \neq \left\vert b\right\vert $ and noncollinear with the
	origin, is determined as $$v_{\mathbb{B}^{2}}\left( a,b\right) =\max \left\{
	\measuredangle (a,w_{1},b),\measuredangle (a,w_{2},b)\right\} \,$$ where $%
	w_{1},w_{2}\in \partial \mathbb{B}^{2}$ lie on the circle $S\left( c,\sqrt{%
		\left\vert c\right\vert ^{2}-1}\right) $. \cite[Theorem 3.2]{fkv} shows that
	the circle $C\left( a,b,w\right) $ with $w\in \left\{ w_{1},w_{2}\right\} $
	is tangent to the unit circle at $w$. In \cite[Theorem 1.1]{fkv} the
	explicit formulas $$\left\{ w_{1},w_{2}\right\} =\left\{ \left( 1\pm i\sqrt{%
		\left\vert c\right\vert ^{2}-1}\right) /\overline{c}\right\} $$ are obtained,
	where $$c=\mathrm{LIS}[a,b,a^{\ast },b^{\ast }]=\frac{\left( \overline{a}-%
		\overline{b}\right) ab+a-b}{\left\vert a\right\vert ^{2}-\left\vert
		b\right\vert ^{2}}=\frac{a\left( 1-\left\vert b\right\vert ^{2}\right)
		-b\left( 1-\left\vert a\right\vert ^{2}\right) }{\left\vert a\right\vert
		^{2}-\left\vert b\right\vert ^{2}}\,.$$ The formula (\ref{tangency}) gives for $w_{1},w_{2}$ the
	same result as above in the case $\left\vert a\right\vert \neq \left\vert
	b\right\vert $, but is applicable in the more general case $a\neq b$.
	Indeed, $$\left( a-b\right) \overline{a}\overline{b}+\overline{a}-\overline{b}%
	=\overline{c}\left( \left\vert a\right\vert ^{2}-\left\vert b\right\vert
	^{2}\right) $$ and (\ref{SIDEN2}) imply $$\left\vert c\right\vert ^{2}-1=%
	\frac{\left\vert a-b\right\vert ^{2}\left( 1-\left\vert a\right\vert
		^{2}\right) \left( 1-\left\vert b\right\vert ^{2}\right) }{\left( \left\vert
		a\right\vert ^{2}-\left\vert b\right\vert ^{2}\right) ^{2}}\,,$$ 
	therefore 
	\[
	\frac{1\pm i\sqrt{\left\vert c\right\vert ^{2}-1}}{\overline{c}}=\frac{%
		\left\vert a\right\vert ^{2}-\left\vert b\right\vert ^{2}\pm i\left\vert
		a-b\right\vert \sqrt{\left( 1-\left\vert a\right\vert ^{2}\right) \left(
			1-\left\vert b\right\vert ^{2}\right) }sign\left( \left\vert a\right\vert
		^{2}-\left\vert b\right\vert ^{2}\right) }{\left( a-b\right) \overline{a}%
		\overline{b}+\overline{a}-\overline{b}}.
	\]
	A simple calculation shows that $w_{1},w_{2}$ lie also on the circle
	centered at $\frac{c}{2}$ passing trough the origin.
\end{nonsec}

\begin{figure}[H]
	\centering
	\begin{subfigure}[b]{0.45\textwidth}
		\centering
		\includegraphics[width=\textwidth]{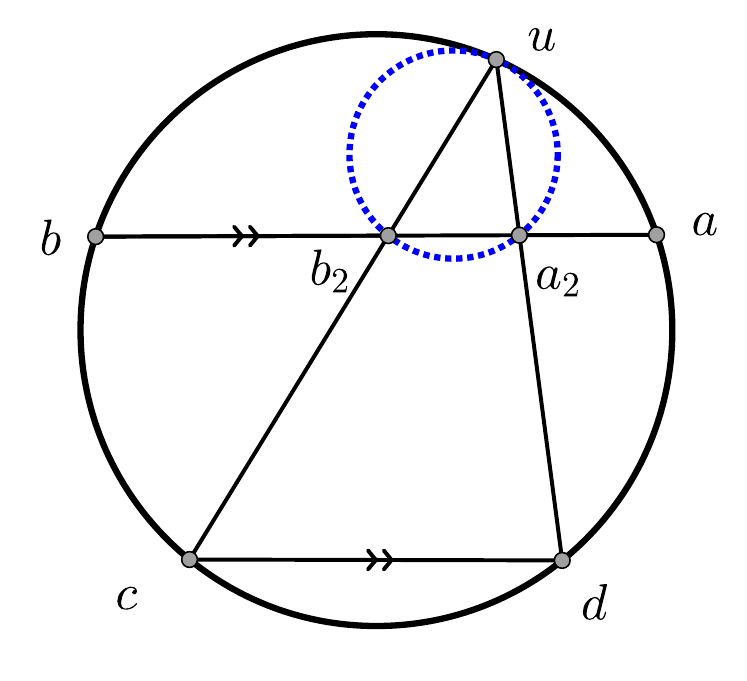}
		\caption{}
	\end{subfigure}
	\hfill
	\begin{subfigure}[b]{0.45\textwidth}
		\centering
		\includegraphics[width=\textwidth]{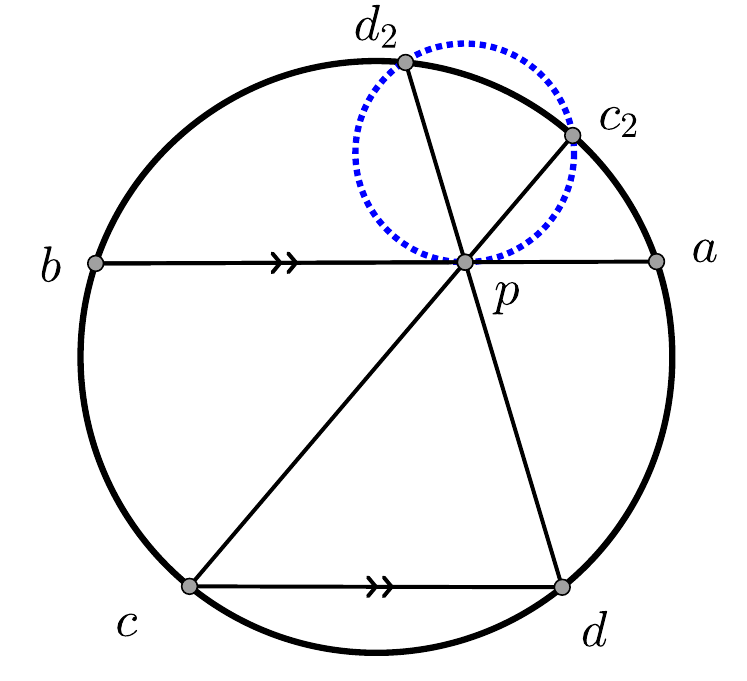}
		\caption{}
		\label{fig:three}
	\end{subfigure}
	\label{fig:four}
	\caption{(A)
		Corollary \ref{LittleThm} (1): If $L[a,b]$ and $L[c,d]$ are parallel, then the circle $C(b_2,a_2,u)$ is tangent to the unit circle at $u$ and \hfill \vspace{0.2cm}
		\centerline{
			$  \exp h_{{\mathbb B}^2}(a_2, b_2)=|a,a_2,b_2, b|=
			\frac{|a-c|^2}{|a-d|^2} \,.$}
		\vspace{0.2cm}
		(B)
		Corollary \ref{LittleThm} (2): If $L[a,b]$ and $L[c,d]$ are parallel and $p \in [a,b]\,,$ then the circle $C(p, c_2,d_2)$ is tangent to the chord $[a,b]$ at the point $p$ and $|a,c_2,d_2,b|=|a,d,c,b| \,.$}
	\label{fig:LittleThm}
\end{figure}
\begin{cor}
	\label{LittleThm}
	Let $a$, $b$, $c$, $d$ \ be points on the unit circle $\partial \mathbb{B}%
	^{2}$, in this order.
	
	(1) Let $u$ be a point on the arc of the unit circle between $a$ and $b$
	that does not contain $c$, $d$ and let $a_{2}=\mathrm{LIS}\left[ a,b,u,d%
	\right] $ and $b_{2}=\mathrm{LIS}\left[ a,b,u,c\right] $. Then 
	\[
	\exp h_{\mathbb{B}^{2}}\left( a_{2},b_{2}\right) =\left\vert \frac{(a-c)(b-d)%
	}{(a-d)(b-c)}\right\vert 
	\]%
	is independent of $u$. Moreover, if the lines $L\left[ a,b\right] $ and $L%
	\left[ c,d\right] $ are parallel, then the circle $C\left(
	a_{2},b_{2},u\right) $ is tangent to the unit circle at the point $u$ and 
	\[
	\exp h_{\mathbb{B}^{2}}\left( a_{2},b_{2}\right) =\left\vert \frac{a-c}{a-d}%
	\right\vert ^{2}. 
	\]
	
	(2) Let $p\in \mathbb{B}^{2}$. For every $z\in \partial \mathbb{B}^{2}$,
	denote by $f_{p}\left( z\right) $ the point where $L[p,z]$ intersects again
	the unit circle. Then 
	\[
	f_{p}\left( z\right) =-\frac{§z-p }{1-\overline{p} z}\text{
		for all }z\in \partial \mathbb{B}^{2}
	\]%
	and  $f_{p}$ is  a M\"{o}bius self-map of the unit disk.
	
	Let $c_{2}=f_{p}\left( c\right) $ and $d_{2}=f_{p}\left( d\right) $. Assume
	that $p$ belongs to the segment $\left[ a,b\right] $. Then \newline
	(i) $\left\vert a,c_{2},d_{2},b\right\vert =\left\vert a,d,c,b\right\vert $.
	\newline
	(ii) The lines $L[a,b]$ and $L[c,d]$ are parallel if and only if the circle $%
	C\left( p,c_{2},d_{2}\right) $ is tangent to the chord $[a,b]$ at the point $%
	p$.
	
	(3) Let $p\in \mathbb{B}^{2}$ be on the segment $\left[ a,b\right] $. Then
	the lines $L[a_{2},c_{2}]$ and $L[b_{2},d_{2}]$ intersect at a point $v$
	that belongs to the unit circle. Moreover, if the lines $L[a,b]$ and $L[c,d]$
	are parallel and if $p$ is the Hilbert midpoint of $a_{2}$ and $b_{2}$, then
	$v=\left( c+d\right) /\left\vert c+d\right\vert $, where $c\neq -d$.
\end{cor}

\begin{proof}
	(1) The claim follows from Theorem \ref{Hilbertmidpoint} (1).

	(2) 
	The equation of the line passing through $ z_1 $ and $ z_2 $ is given by
	$$
	(\overline{z_1}-\overline{z_2})z-(z_1-z_2)\overline{z}
	=\overline{z_1}z_2-z_1\overline{z_2}.
	$$
	So if $ z_1,z_2\in\partial\mathbb{B}^2 $, the above equation is written as
	$$
	z+z_1z_2\overline{z}=z_1+z_2.
	$$
	Since $ z,f_p(z)\in\partial\mathbb{B}^2 $ and the point $ p $ is on
	$ L[z,f_p(z)] $, the following holds
	$$
	p+z f_p(z)\overline{p}=z+f_p(z),
	$$
	that is, 
	$$
	f_p(z)=-\frac{z-p}{1-\overline{p}z}=e^{\pi i} \frac{z-p}{1-\overline{p}z} \,.
	$$
	Therefore $ f_p $ is an automorphism on $ \mathbb{B}^2 $.
	
	(i) As M\"obius transformation $ f_p $ preserves cross-ratios,
	$$
	|a,d,c,b|=|a,c_2,d_2,b|.
	$$
	
	\medskip
	(ii) The lines ${L[a,b]}$ and $L[c,d]$ are parallel if and only if $%
	\measuredangle (c_{2},p,a)=\measuredangle (p,c,d)$. The angles $%
	\measuredangle (p,c,d)$ and $\measuredangle (c_{2},d_{2},p)$ are always
	equal. Therefore, ${L[a,b]}$ and $L[c,d]$ are parallel if and only if $%
	\measuredangle (c_{2},p,a)=\measuredangle (c_{2},d_{2},p)$. On the other
	hand, $\measuredangle (c_{2},p,a)=\measuredangle (c_{2},d_{2},p)$ if and
	only if the circle $C(p,c_{2},d_{2})$ and the line ${L}[a,b]$ are tangent at 
	$p$. 
	\medskip
	\newline
	

	\begin{figure}[H]
		\centering
		\begin{subfigure}[b]{0.45\textwidth}
			\centering
			\includegraphics[width=\textwidth]{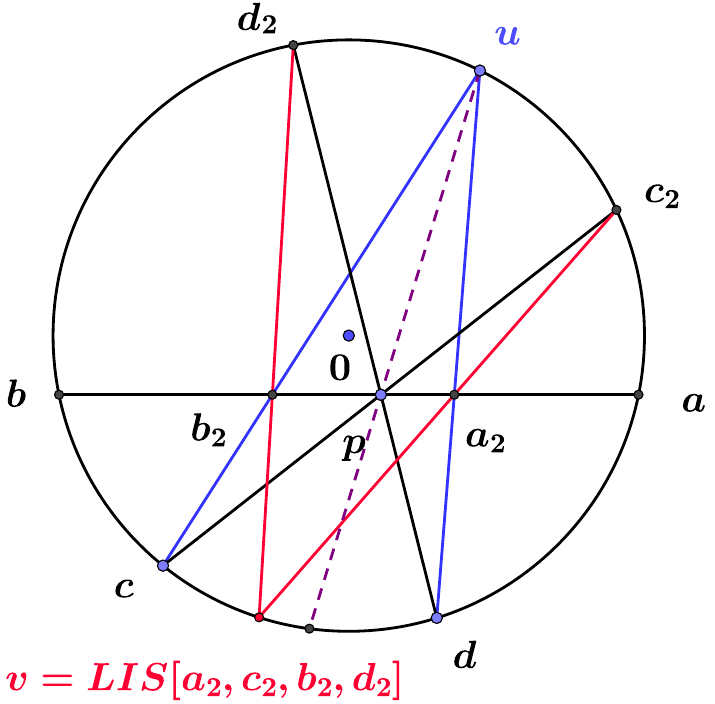}
			\caption{}
		\end{subfigure}
		\hfill
		\begin{subfigure}[b]{0.45\textwidth}
			\centering
			\includegraphics[width=\textwidth]{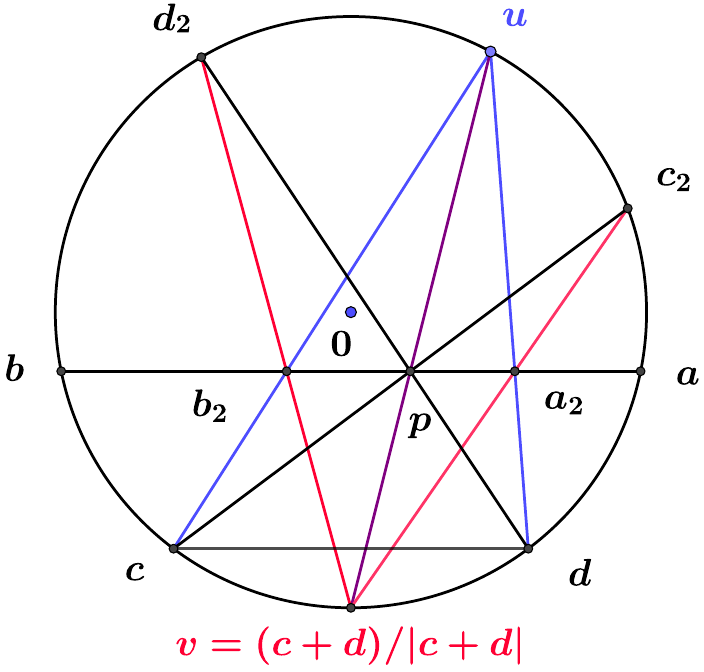}
			\caption{}
			\label{fig:cor325c}
		\end{subfigure}
		\label{fig:cor325a}
		\caption{(A)
			Corollary 3.25 (3): The case $L[a,b]$ and $L[c,d]$ are nonparallel.
			(B)
			Corollary 3.25 (3): The case $L[a,b]$ and $L[c,d]$ are parallel.}
		\label{fig:Cor325b}
	\end{figure}
	
	(3) The line $L\left[ a_{2},c_{2}\right] $ intersects the unit circle at two
	points: $c_{2}$ and, say, $v$. We prove that $v\in L[b_{2},d_{2}]$, hence $v=%
	\mathrm{LIS}\left[ a_{2},c_{2},b_{2},d_{2}\right] $. \newline
	Let $b_{3}=\mathrm{LIS}[a,b,v,d_{2}]$. By Lemma \ref{lem:cross-ratio} (1), $%
	|a,a_{2},b_{3},b|=|a,c_{2},d_{2},b|$. By (2) and (1) in this corollary and
	the definition of the Hilbert distance, 
	\[
	|a,c_{2},d_{2},b|=|a,d,c,b|=\exp h_{\mathbb{B}^{2}}\left( a_{2},b_{2}\right)
	=|a,a_{2},b_{2},b|.
	\]%

	Then \[|a,a_{2},b_{3},b|=|a,a_{2},b_{2},b|\,,
	\quad {\rm i.e.} \quad \frac{\left\vert
		a-b_{3}\right\vert }{\left\vert a-a_{2}\right\vert }\frac{\left\vert
		b-a_{2}\right\vert }{\left\vert b-b_{3}\right\vert }=\frac{\left\vert
		a-b_{2}\right\vert }{\left\vert a-a_{2}\right\vert }\frac{\left\vert
		b-a_{2}\right\vert }{\left\vert b-b_{2}\right\vert }\,,\] 
	hence $\frac{%
		\left\vert a-b_{3}\right\vert }{\left\vert b-b_{3}\right\vert }=\frac{%
		\left\vert a-b_{2}\right\vert }{\left\vert b-b_{2}\right\vert }$. As $b_{3}$
	and $b_{2}$ divide the segment $\left[ a,b\right] $ in the same ratio, we
	obtain $b_{3}=b_{2}$, therefore $v\in L[b_{3},d_{2}]=L[b_{2},d_{2}].$
	
	Now assume that the lines $L[a,b]$ and $L[c,d]$ are parallel and $p$ is the
	Hilbert midpoint of $a_{2}$ and $b_{2}$. Denote $m=\left( c+d\right)
	/\left\vert c+d\right\vert $. We have to prove that $c_{2}$, $a_{2}$ and $m$
	are collinear. Similarly, it can be proved that $d_{2}$, $b_{2}$ and $m$ are
	collinear. Then $m=\mathrm{LIS}\left[ a_{2},c_{2},b_{2},d_{2}\right] =v$. 
	\newline
	By Theorem \ref{Hilbertmidpoint} (2), $p=\mathrm{LIS}\left[ a,b,m,u\right] $%
	. Let $a_{3}=\mathrm{LIS}[a,b,c_{2},m]$. We prove that $a_{3}=a_{2}$. By
	Lemma \ref{lem:cross-ratio} (1), using the projection from $c_{2}$ we get 
	\[
	\left\vert a,a_{3},p,b\right\vert =\left\vert a,m,c,b\right\vert .
	\]%
	Using the projection from $u$ and Lemma \ref{lem:cross-ratio} (1), it
	follows that 
	\[
	\left\vert a,m,c,b\right\vert =\left\vert a,p,b_{2},b\right\vert .
	\]%
	As $p$ is the Hilbert midpoint of $a_{2}$ and $b_{2}$, 
	\[
	\left\vert a,p,b_{2},b\right\vert =\left\vert a,a_{2},p,b\right\vert .
	\]%
	The latter three equalities imply $\left\vert a,a_{3},p,b\right\vert
	=\left\vert a,a_{2},p,b\right\vert $, hence $a_{3}$ and $a_{2}$ divide the
	segment $\left[ a,b\right] $ in the same ratio, therefore $a_{3}=a_{2}$.
\end{proof}

%


%

\section{Hilbert circles and spheres}

\begin{nonsec}{\bf Hilbert disks and circles.} For a bounded convex domain $D$ and a point $z \in D$ the set
\[B_h(z,t) = \{w \in D: h_D(z,w)<t\} \,, t >0,\]
is called the Hilbert disk with center $z$ and radius $t\,.$ 

Busemann \cite[\S 18]{bu} has proved
that Hilbert disks are convex sets, see also \cite[p. 170]{p}. Moreover, he has also demonstrated that ellipses
in Hilbert geometry are not convex sets \cite[p. 170]{p}.

For the case when $D$ is a triangle, 
Papadopoulos \cite[p. 168, Fig. 5.10]{p}
explicitly describes Hilbert circles as
hexagons of Euclidean geometry.

The next two figures display these disks
in a polygonal domain and in a circular sector domain, resp. It appears that the points of non-smoothness of $\partial D$ have influence on the Hilbert circle $\partial B_h(z,t).$
\end{nonsec}

\begin{figure}[H]
	\centering
	\includegraphics[width=0.55\linewidth]{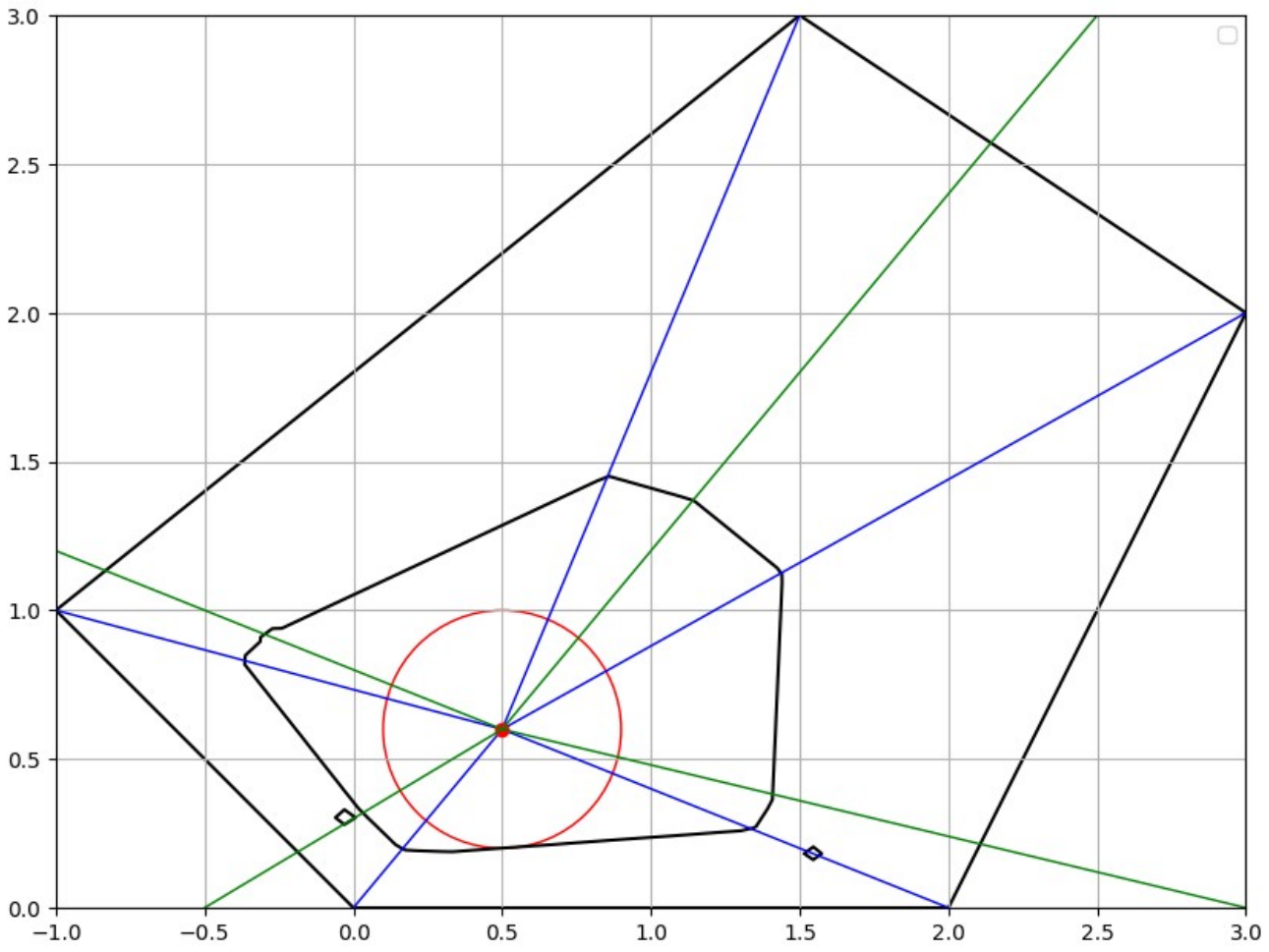}
	\caption{A Hilbert disk $B_h(z,t)$ and a maximal Euclidean circle with the same center $z$
	in a bounded convex polygonal domain. The points of non-smoothness of $\partial B_h(z,t)$ seem
	to be situated on the lines through vertices and the center $z \,.$
	}
	\label{figHdisk}
\end{figure}

\begin{figure}[H]
	\centering
\includegraphics[width=0.4\linewidth]{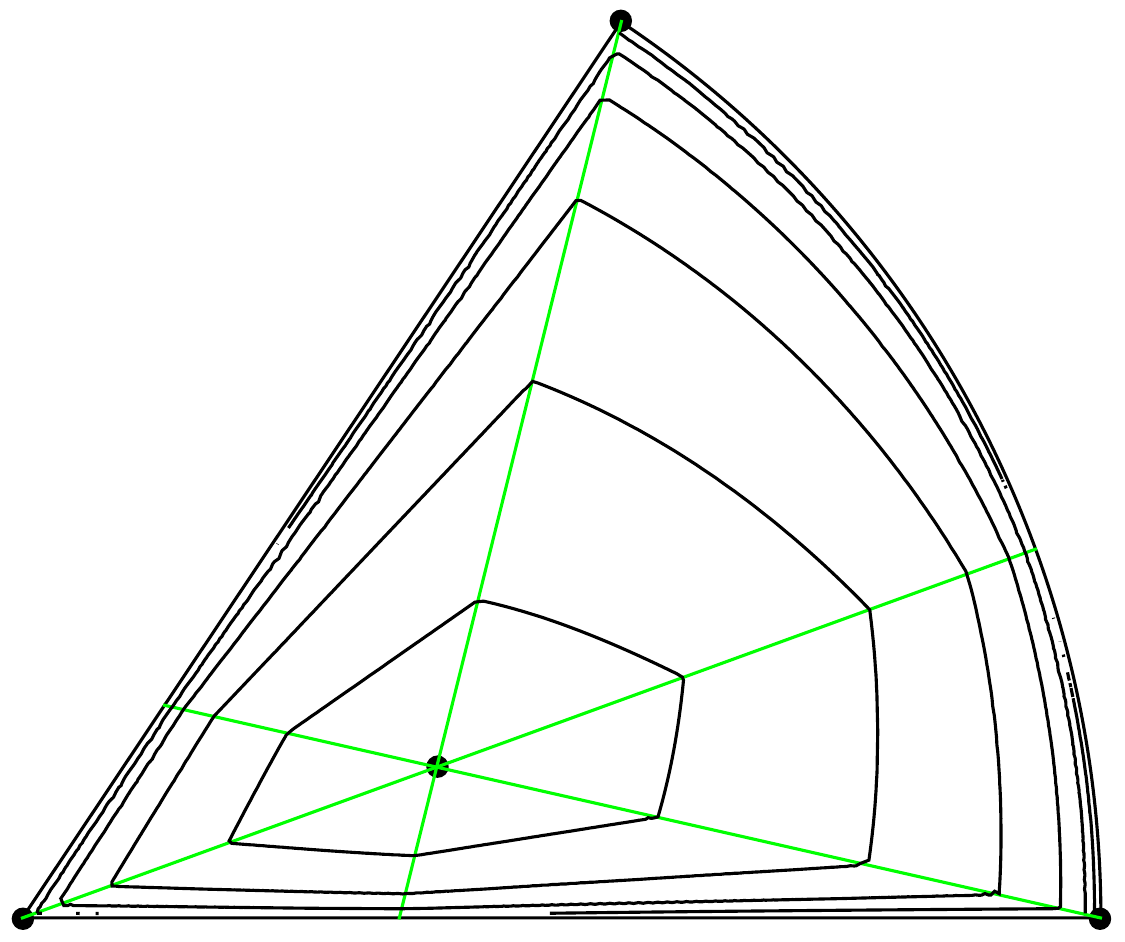}
	\caption{Some Hilbert circles centered at a point of a convex sector with different radii.
	}
	\label{fig:fig5}
\end{figure}

\begin{nonsec}{\bf Hilbert spheres in n-dimensional space.}
Let $\mathbb{B}^{n}$ be the unit ball in the Euclidean space $\mathbb{R}^{n}$
and $h=h_{\mathbb{B}^{n}}$ be the Hilbert distance in $\mathbb{B}^{n}$.

We prove that each sphere $S_{h}\left( c,R\right) =\left\{ x\in \mathbb{B}%
^{n}:h_{\mathbb{B}^{n}}\left( x,c\right) =R\right\} $, with $c\in \mathbb{B}%
^{n}$ and $R>0$, is an ellipsoid of revolution, with the center collinear
with $0$ and $c$, having the minor semi-axis along $c$ and the other
semi-axes equal to each other. This result generalizes \cite[Theorem 6.1]{afv}.
\end{nonsec}

\begin{thm}
\label{Basphere}Let $\mathbb{B}^{n}$ be the unit ball in the Euclidean space 
$\mathbb{R}^{n}$ and $h=h_{\mathbb{B}^{n}}$ be the Hilbert distance in $%
\mathbb{B}^{n}$. The sphere $S_{h}\left( c,R\right) =\left\{ x\in \mathbb{B}%
^{n}:h_{\mathbb{B}^{n}}\left( x,c\right) =R\right\} $, with $c\in \mathbb{B}%
^{n}$ and $R>0$, is an ellipsoid of revolution, with the center 
\[
c^{\prime }=\frac{c}{ {\ch} ^{2}\left( \frac{R}{2}\right) -\left\vert
c\right\vert ^{2} {\sh} ^{2}\left( \frac{R}{2}\right) } \,\text{.} 
\]%
Write
\[
a_{\min }=\frac{1}{2}\frac{%
\left( 1-\left\vert c\right\vert ^{2}\right)  {\sh} (R)}{ {\ch} ^{2}\left( 
\frac{R}{2}\right) -\left\vert c\right\vert ^{2} {\sh} ^{2}\left( \frac{R}{2}%
\right) }\,,\quad
a_{\max }=\frac{ {\sh} (\frac{R%
}{2})\sqrt{1-\left\vert c\right\vert ^{2}}}{\sqrt{ {\ch} ^{2}\left( \frac{R}{2%
}\right) -\left\vert c\right\vert ^{2} {\sh} ^{2}\left( \frac{R}{2}\right) }} \,.
\]
This ellipsoid  has the minor semi-axis, along $c$, equal to  $a_{\min }$ and every other semi-axis equal to $a_{\max }\,.$%
If $c=0$, then $S_{h}\left( c,R\right) $ is the Euclidean sphere centered
at $c$ with radius $\th \left( \frac{R}{2}\right) $.
\end{thm}

\begin{proof}
For $x\in \mathbb{B}^{n}$, the explicit formula of the Hilbert metric in the
unit ball shows that $h_{\mathbb{B}^{n}}\left( x,c\right) =R$ if and only if 
\begin{equation}\label{NSC}
\left( 1-c\cdot x\right) ^{2}= {\ch} ^{2}\left( \frac{R}{2}\right) \left(
1-\left\vert c\right\vert ^{2}\right) \left( 1-\left\vert x\right\vert
^{2}\right) .  
\end{equation}%
Let $\alpha := {\ch} ^{2}\left( \frac{R}{2}\right) \left( 1-\left\vert
c\right\vert ^{2}\right) $. Then $\alpha >0$.

If $c=0$, then (\ref{NSC}) reduces to $\left\vert x\right\vert =\th
\left( \frac{R}{2}\right) $, that implies $x\in \mathbb{B}^{n}$. \newline
Assume that $c\in \mathbb{B}^{n}\setminus \left\{ 0\right\} $. Using the
matrix notation of vectors in $\mathbb{R}^{n}$, $c\cdot x=c^{T}x=x^{T}c$, we
write (\ref{NSC})\ under the equivalent form 
\[
x^{T}(\alpha I_{n}+cc^{T})x-2c^{T}x+1-\alpha =0. 
\]%
The above equation defines an $\left( n-1\right) $-dimensional ellipsoid, as
the matrix $\alpha I_{n}+cc^{T}$ is positive definite, with the eigenvalues $%
\lambda _{1}=\alpha +\left\vert c\right\vert ^{2}$ and $\lambda _{k}=\alpha $
for $k=2,...,n$ and an orthonormal basis of corresponding eigenvectors $%
v_{1}=\frac{c}{\left\vert c\right\vert }$ and $v_{k}$ with $k=2,...,n$.

In order to obtain a canonical equation of the above ellipsoid, decompose $%
x=\lambda \frac{c}{\left\vert c\right\vert }+z$, where $z$ is orthogonal to $%
c$. These conditions hold if and only if $\lambda =\frac{c\cdot x}{%
\left\vert c\right\vert }$ and $z=x-\frac{c\cdot x}{\left\vert c\right\vert
^{2}}c$. Then $c\cdot x=\left\vert c\right\vert \lambda $ and $\left\vert
x\right\vert ^{2}=\lambda ^{2}+\left\vert z\right\vert ^{2}$. Substituting
in (\ref{NSC}) we get $(1-\left\vert c\right\vert \lambda )^{2}=\alpha
\left( 1-\lambda ^{2}-\left\vert z\right\vert ^{2}\right) $, i.e. 
\[
\left( \left\vert c\right\vert ^{2}+\alpha \right) \lambda ^{2}-2\left\vert
c\right\vert \lambda +\alpha \left\vert z\right\vert ^{2}=\alpha -1\text{.} 
\]%
The equation (\ref{NSC}) is equivalent to 
\begin{equation}\label{Elfin}
\frac{\left( \lambda -\frac{\left\vert c\right\vert }{\left\vert
c\right\vert ^{2}+\alpha }\right) ^{2}}{\left( \frac{\sqrt{\alpha \left(
\left\vert c\right\vert ^{2}+\alpha -1\right) }}{\left\vert c\right\vert
^{2}+\alpha }\right) ^{2}}+\frac{\left\vert z\right\vert ^{2}}{\left( \sqrt{%
\frac{\left\vert c\right\vert ^{2}+\alpha -1}{\left\vert c\right\vert
^{2}+\alpha }}\right) ^{2}}=1,  
\end{equation}%
which is the equation of an $\left( n-1\right) $-dimensional ellipsoid $\mathcal{E}$ with the center at $\frac{c}{\left\vert c\right\vert
^{2}+\alpha }$, having the minor semi-axis $a_{\min }=\frac{\sqrt{\alpha
\left( \left\vert c\right\vert ^{2}+\alpha -1\right) }}{\left\vert
c\right\vert ^{2}+\alpha }$ and the other semi-axes equal to $a_{\max }=%
\sqrt{\frac{\left\vert c\right\vert ^{2}+\alpha -1}{\left\vert c\right\vert
^{2}+\alpha }}$.

After some calculations we get \[a_{\min }=\frac{1}{2}\frac{\left(
1-\left\vert c\right\vert ^{2}\right)  {\sh} (R)}{ {\ch} ^{2}\left( \frac{R}{2}%
\right) -\left\vert c\right\vert ^{2} {\sh} ^{2}\left( \frac{R}{2}\right) }\,  
 \quad \quad
{\rm and}\quad \quad a_{\max }=\frac{ {\sh} (\frac{R}{2})\sqrt{1-\left\vert c\right\vert ^{2}}%
}{\sqrt{ {\ch} ^{2}\left( \frac{R}{2}\right) -\left\vert c\right\vert
^{2} {\sh} ^{2}\left( \frac{R}{2}\right) }}\,.\]

We proved that $S_{h}\left( c,R\right) \subset \mathbb{B}^{n}\cap \mathcal{E}
$. Conversely, we check that $\mathcal{E}\subset S_{h}\left( c,R\right) $.
Let $x\in \mathcal{E}$. Using the decomposition $x=\lambda \frac{c}{%
\left\vert c\right\vert }+z$, with $\lambda \in \mathbb{R}$ and $z\in 
\mathbb{R}^{n}$ orthogonal to $c$, we see that $\lambda $ and $z$ satisfy (%
\ref{Elfin}), hence (\ref{NSC}) holds. The right-hand side of (\ref{NSC}) is
non-negative, therefore $x\in \mathbb{B}^{n}$. Moreover, (\ref{NSC}) shows
that $h_{\mathbb{B}}\left( c,x\right) =R$. In conclusion, $S_{h}\left(
c,R\right) =\mathcal{E}$.
\end{proof}

\begin{rem}
Using the notations $s=\left\vert c\right\vert $ and $k={\th} \left( \frac{R%
}{2}\right) $, the formulas from Theorem \ref{Basphere} can be written as
follows: 
\[
c^{\prime }=\frac{1-k^{2}}{1-s^{2}k^{2}}c\text{, } \quad a_{\min }=\frac{k\left(
1-s^{2}\right) }{1-s^{2}k^{2}}\text{,}\quad a_{\max }=\frac{k\sqrt{1-s^{2}}}{\sqrt{%
1-s^{2}k^{2}}}.\text{ } 
\]%
Rotational symmetry of the ellipsoid $\mathcal{E}$ around the line $L[0,c]$ explains why all semi-axes orthogonal to the vector $c$ are equal.
\end{rem}

%
\section{H\"older continuity}
\begin{nonsec}{\bf The special function
$\varphi_K(r)$ of the Schwarz lemma.}
Let $F_0, F_1$ be non-empty subsets of $\mathbb{R}^n$ and denote the family of all closed non-constant curves joining $F_0$ and $F_1$ in $\mathbb{R}^n$ by $\Delta(F_0, F_1; \mathbb{R}^n)$. The Gr\"otzsch capacity is the decreasing homeomorphism defined as $\gamma_n : (1,\infty) \to (0,\infty)$,
\[
\gamma_n(s) = {\M}(\Delta(\mathbb{B}^n, [se_1, \infty]; \mathbb{R}^n)), \quad s > 1,
\]
where $e_1$ is the first unit vector and ${\M}$ stands for the conformal modulus (see \cite[7.17]{hkv}, p. 121). For the definition and more details about the conformal modulus, see \cite{hkv}, pp. 103-131. In the special case $n=2$, we have the explicit formulas \cite{hkv}, (7.18), p. 122
\[
\gamma_2(1/r) = \frac{2\pi}{\mu(r)}, \quad \mu(r) = \frac{\pi}{2} \frac{\K(\sqrt{1 - r^2})}{\K(r)}, \quad \K(r) = \int_0^1 \frac{dx}{\sqrt{(1-x^2)(1 - r^2x^2)}}
\]
with $0<r<1$.

\medskip
\begin{nonsec}{\bf K-quasiregular mappings.} \cite[p. 288]{hkv}. 
	Let $D\subset\mathbb{R}^n$ be a domain. A mapping
	$f:D\to\mathbb{R}^n$ is said to be quasiregular if $f$ is
	$\operatorname{ACL}^n$ and there exists $K\geq 1$ such that
	\[
	|f^\prime(x)|^n \leq K J_f(x)
	\]
	for almost every $x\in D$. The smallest $K$ for which this condition
	holds is the outer dilatation $K_O(f)$ of $f$.
	
	If $f$ is quasiregular, then there exists $K^\prime\geq 1$ such that
	\[
	J_f(x)\leq K^\prime\bigl(l(f^\prime(x))\bigr)^n
	\]
	for almost every $x\in D$, and the smallest such $K^\prime$ is the inner dilatation $K_I(f)$ of $f$.
	
	The maximal dilatation is
	\[
	K(f)=\max\{K_I(f),K_O(f)\}.
	\]
	A $K$-quasiregular mapping $f$ is a quasiregular mapping satisfying
	\[
	K(f)\leq K.
	\]
\end{nonsec}

The function $\varphi_K : [0,1] \to [0,1]$, will occur in the quasiregular version of the Schwarz lemma as well as in its many applications when $n=2$. This function is defined as follows \cite[Eq (9.13), p. 167]{hkv}. For $0 < r < 1$ and $K > 0$ we define the special function
\begin{equation*}
	\varphi_K(r) = \frac{1}{\gamma_2^{-1}(K \gamma_2(1/r))} 
\end{equation*}
and set $\varphi_K(0) = 0$, $\varphi_K(1) = 1$. It is easy to see that $\varphi_K : [0,1] \to [0,1]$ is a homeomorphism.
\end{nonsec}

\begin{thm} (See \cite[Thm 16.39, p.313]{hkv}) \label{2.12} 
	If \( f : \mathbb{B}^2 \to \mathbb{B}^2 \) is a non-constant \( K \)-quasiregular mapping, then
	\[
	\rho_{\mathbb{B}^2}(f(a), f(b)) \leq c(K) \max \left\{ \rho_{\mathbb{B}^2}(a, b), \rho_{\mathbb{B}^2}(a, b)^{1/K} \right\}
	\]
	for all \( a, b \in \mathbb{B}^2 \), where \( c(K) = 2\,\mathrm{arth}(\varphi_K(\mathrm{th}\tfrac{1}{n})) \), and
	\[
	K \leq u(K - 1) + 1 \leq \log(\mathrm{ch}(K\,\mathrm{arch}(e))) \leq c(K) \leq v(K - 1) + K
	\]
	with \( u = \mathrm{arch}(e)\,\mathrm{th}(\mathrm{arch}(e)) > 1.5412 \) and \( v = \log\left(2(1 + \sqrt{1 - 1/e^2})\right) < 1.3507 \). In particular, \( c(1) = 1 \). Moreover, if 
\( f : \mathbb{B}^2 \to D \) is a non-constant \( K \)-quasiregular mapping and
$D$ is conformally equivalent to \(\mathbb{B}^2 \), then 	
\[
	\rho_{D}(f(a), f(b)) \leq c(K) \max \left\{ \rho_{\mathbb{B}^2}(a, b), \rho_{\mathbb{B}^2}(a, b)^{1/K} \right\}
	\]
for all \( a, b \in \mathbb{B}^2 \).	
\end{thm}

\begin{proof}
The first part is proven in \cite[Thm 16.39, p.313]{hkv} and the second part follows from the conformal invariance
of the hyperbolic metric.
\end{proof}

\begin{nonsec}{\bf Apollonian metric.}\label{apomet}
We turn our attention to the Apollonian metric \cite{b2,b3}. Let \( D \subset \mathbb{R}^n \) be a domain.  
The {\it Apollonian metric} \( \alpha_D(a, b) \) between two points \( a, b \in D \) is defined as \cite{gh}
\[
\alpha_D(a, b) = \sup_{x, y \in \partial D} \log \frac{|x - b||y - a|}{|x - a||y - b|}.
\]
Trivially, for convex bounded domains $D,$ we have
$h_D \le \alpha_D.$
It is well-known \cite{b2} that $\rho_{\mathbb{B}^2} = \alpha_{\mathbb{B}^2}$ and   $\rho_{\mathbb{H}^2} = \alpha_{\mathbb{H}^2}.$
\end{nonsec}

\begin{thm} \cite[Thm 6.1]{b2} \label{Bthm} Let $D$ be
a bounded convex plane domain. Then for all $z,w \in D$
\[
\frac{1}{2}\,\alpha_D(z,w) \le \rho_D(z,w) \le 4 \sh(\frac{1}{2}\, \alpha_D(z,w))\,.
\]
\end{thm}

\begin{nonsec}{\bf Open problem.} \label{hilApo}
We do not know whether the lower bound
\(\frac{1}{2}\,\alpha_D(z,w)\) above
could be replaced with \(\alpha_D(z,w)\).
\end{nonsec}

The {\it M\"obius metric}
of a domain $G \subset \overline{ \mathbb{R}}^n$ with ${\rm card} ( \overline{ \mathbb{R}}^n \setminus G) \ge 2,$ is defined  for $a,b \in G$
\begin{equation}\label{ferrmob}
\delta_G(a,b)= \log \left( 1+\sup_{u,v \in \partial G}|u,a,v,b|\right) 
\end{equation}
and by \cite[Thm 3.11]{se}
\[
\alpha_G \le \delta_G \le \log(e^{\alpha_G}+2) \,.
\]
For $a,b \in { \mathbb{B}^n}$
\begin{equation}\label{ferrmob2}
\rho_{ \mathbb{B}^n}(a,b) =\delta_{ \mathbb{B}^n}(a,b) \,.
\end{equation}
For further information about this metric, see \cite{se} and
the literature cited in  \cite[pp.73-76]{hkv}.

\begin{nonsec}{\bf Open problem.}
It seems to be an open problem to find explicit
formulas for  $\delta_{G}(a,b)$
  in the case of simple domains like strip, sector, convex polygon.
\end{nonsec}

\begin{nonsec}{\bf Proof of Theorem \ref{qrHolder}.}
It follows easily from the definition of the Apollonian metric that

\[
h_D(f(a), f(b)) \leq \alpha_D(f(a), f(b)).
\]
On the other hand, by Theorem \ref{Bthm}
we can write
\[
 h_D(f(a), f(b)) \leq \alpha_D(f(a), f(b))\leq 2 \, \rho_D(f(a), f(b)).
\]
By Theorem \ref{2.12} and Corollary \ref{hrho}, we get
\begin{align*}
	h_D(f(a), f(b)) 
	&\leq \alpha_D(f(a), f(b)) \leq 2 \, \rho_D(f(a), f(b)) \\
	&\leq 2\,c(K) \max \left\lbrace \rho_{\mathbb{B}^2}(a, b), \left(\rho_{\mathbb{B}^2}(a, b)\right)^{1/K} \right\rbrace \\
	&\leq 2\,c(K) \max \left\lbrace \frac{h_{\mathbb{B}^2}(a, b)}{\sqrt{1 - m^2}}, \left(\frac{h_{\mathbb{B}^2}(a, b)}{\sqrt{1 - m^2}}\right)^{1/K} \right\rbrace \\
	&\leq \frac{2\,c(K)}{\sqrt{1 - m^2}} \max \left\lbrace h_{\mathbb{B}^2}(a, b), \left(h_{\mathbb{B}^2}(a, b)\right)^{1/K} \right\rbrace.
\end{align*}
The claim follows. \hfill $\square$
\end{nonsec}

\begin{nonsec}{\bf Remark.} \label{sharpness}
 If the open problem \ref{hilApo} could be solved in the affirmative, then the constant
$2 c(K)$ in Theorem
\ref{qrHolder} could be replaced with $ c(K)\,.$ 
%

\end{nonsec}

%



\textbf{Acknowledgements.} \c{S}. Alt\i nkaya is supported by the Scientific and Technological Research Council of T\"{u}rkiye (TUBITAK), Project Number: 1059B192402218. \\
M. Fujimura is partially supported by JSPS KAKENHI Grant Number JP25K07039.

\bigskip

\end{document}